\documentclass[12pt,a4paper]{article}

\usepackage{a4wide}
\usepackage{amsmath,amscd,amsfonts,amssymb,amsthm,color,epsfig,float,graphicx,latexsym,makeidx,psfrag,subfigure}
\usepackage[
  bookmarksopen,
  colorlinks,
  linkcolor = blue,
  urlcolor  = red,
  citecolor = red,
  menucolor = blue
]{hyperref}

\newtheorem{theo}{Theorem}

\newtheorem{lemma}[theo]{Lemma}
\newtheorem{pr}[theo]{Proposition}
\newtheorem{df}{Definition}
\newtheorem{remark}{Remark}
\newtheorem{ex}[theo]{Example}
\newtheorem{assumption}[theo]{Assumption}

\def\a{a}
\def\ud{u^\delta}
\def\argmin{\text{argmin}}

\def\F{\mathcal{F}}
\newcommand\N{\mathbb{N}}

\def\a{a}
\def\ud{u^\delta}
\def\argmin{\text{argmin}}
\def\X{H^{1+\varepsilon}(D)}
\def\Y{L^2(D)}
\newcommand\norm[2][]{{\left\lVert#2\right\rVert_{#1}}}

\newcommand\domain[1]{{\cal D}({#1})}
\newcommand\KL[2]{{\text{KL}}({#1},{#2})}
\title{\bf Convex Regularization of Local Volatility Estimation in a Discrete Setting}

\author{V. Albani\thanks{IMPA, Estr. D. Castorina
         110, 22460-320 Rio de Janeiro, Brazil, \href{mailto:vvla@impa.br}{\tt
         vvla@impa.br}}, \,
        A. De Cezaro\thanks{ IMEF/FURG,
         Av. Italia Km 8, 96201-900 Rio  Grande, Brasil,\, \href{mailto:decezaro@impa.br}
        {\tt decezaro@impa.br}}\,
                  and \
        J. P.~Zubelli\thanks{IMPA, Estr. D. Castorina 110,
         22460-320 Rio de Janeiro, Brasil,\, \href{mailto:zubelli@impa.br}
        {\tt zubelli@impa.br}}
        }
\begin{document}
\maketitle

\newcommand{\slugmaster}{%
\slugger{}{2013 Albani, De Cezaro, Zubelli}{}{}{}}
% \renewcommand{\thefootnote}{\fnsymbol{footnote}}
% \footnotetext[1]{IMPA, Estr. D. Castorina 110, 22460-320 Rio de Janeiro, Brazil, \email{vvla@impa.br}}
% \footnotetext[2]{IMEF/FURG, Av. Italia km 8, 96201-900 Rio  Grande, Brasil,\, \email{decezaro@impa.br}}
% \footnotetext[3]{IMPA, Estr. D. Castorina 110, 22460-320 Rio de Janeiro, Brasil,\, \email{zubelli@impa.br}}
% \renewcommand{\thefootnote}{\arabic{footnote}}

%%%%%%%%%%%%%%%%%%%%%%%%%%%%%%%%%%%%%%%%%%%%%%%%%%%%%%%%%%%%%%%%%%%%%%%%%%%%%%%%%%%%%%%%%%%
\begin{abstract}
We apply convex regularization techniques to the problem of calibrating the local volatility surface model of Dupire taking into account the practical requirement of discrete grids and noisy data. 
Such requirements are the consequence of bid and ask spreads, quantization of the quoted prices and lack of liquidity of option prices for strikes far way from the at the money level. 

We obtain convergence rates and results comparable to those obtained in the idealized continuous setting. 
Our results allow us to take into account separately the uncertainties due to the price noise and those due to discretization errors. Thus allowing better discretization levels both in the domain and in the image of the parameter to solution operator.

We illustrate the results with simulated as well as real market data. We also validate the results by comparing the implied volatility prices of market data with the computed prices of the calibrated model.
\vspace{10pt}

\noindent {\bf Keywords:} Convex regularization,
local volatility surfaces, 
regularization convergence rates, 
numerical methods for volatility calibration.
\vspace{10pt}

\noindent {\bf AMS subject classifications:} 91G60 (65M32 91B70)
\end{abstract}

% \begin{AMS}
% Convex regularization,
% local volatility surfaces, 
% regularization convergence rates, 
% numerical methods for volatility calibration.
% \end{keywords}
% 
% \begin{AMS}
% 91G60 (65M32 91B70)
% \end{AMS}

\pagestyle{myheadings}
\thispagestyle{plain}
\markboth{V. ALBANI, A DE CEZARO AND J.~P. ZUBELLI}{Local Volatility Estimation in a Discrete Setting}

%-----------------------------------------------------------------------------------------------
\section{Introduction} \label{sec:introd}
During the last two decades we have witnessed an explosion of
interest on mathematical models in Finance. Indeed, the need for
tools to understand risk and volatility in equity and commodity
prices is crucial for the financial industry.
A well accepted class of models consists of an extension of the Black-Scholes 
model known as local-volatility models. It was pioneered
by B.~Dupire in \cite{Dupre-1994} for the continuous case and has
been a standard in many financial applications.

Local volatility models subsume that the volatility coefficient in the stochastic differential equation that determines the dynamics of the underlying is a function of time and of the underlying at that time. The calibration of such surface becomes fundamental for the consistent pricing of more complex derivatives.

In this article we are concerned with the calibration of
local-volatility models from option market data in a realistic setting that 
incorporates the availability of noisy and discrete financial data. 
Mathematically, it is connected to the identification of a (nonconstant) diffusion
coefficient in a parabolic equation. It is well known that such
calibration problem, as many important ones in Mathematical Finance, is highly
ill-posed. In particular, small changes and noise in the data may
lead to substantial changes in the results. Yet, good volatility
surface calibration is crucial in a plethora of applications,
including risk management, hedging, and the evaluation of exotic
derivatives. In order to tackle the ill-posedness of the calibration
we make use of (non quadratic) Tikhonov-type regularization
techniques and extend previous work \cite{DeCezaroScherzerZubelli09,DCZ2011}.
More precisely, we extend the results presented in \cite{DeCezaroScherzerZubelli09,DCZ2011} for a discrete setting, under less restrictive assumptions than the ones used in \cite{PRS2010}.

The introduction of regularization techniques in order to 
stabilize the problem leads to a crucial question: If the noise in
the data goes to zero, does the corresponding regularized solution
converge to the true volatility? If this is the case, it would be
also natural to inquire about the rate of convergence. 

In the sequel, we shall respond such questions by recalling some results present in literature \cite{DCZ2010,DCOZ2010, DeCezaroScherzerZubelli09}. Such results concern existence and stability of the regularized solutions as well as convergence and convergence rate in terms of the noise level.

%%%%%%%%%%%%%%%%%%%%%%%%
\paragraph{The novelty and main contributions of the present article}

While keeping the underlying model in a context of partial differential equations and continuous variables, we apply a discretization process that determines the level of regularization in the coarse grid associated to the available data. This brings the problem from a context of infinite dimensional identification to a finite dimensional one in a regularized way. 
At this point, our contribution is the proof of results concerning the regularizing properties of finite dimensional solutions, the noise level and the levels of discretization in the range and in the domain of the forward operator. 
In particular, the convergence rate result, under this setting, shows that there is an intrinsic relation between the level of discretization in the domain and the noise of the observed prices. 
More precisely, given the noise level and the level of discretization in the range of the forward operator, there is an optimal level of
discretization in the domain $\mathcal{D}(F)$ of $F$ that minimizes the distance between the regularized solution and the true local volatility surface.

It can be intuitively interpreted as follows:\\
{\em If we use very few elements in the
discretization of $\domain{F}$, the resulting surface could not
capture the variability of the original local volatility, i.e., the regularized solution would be almost constant. On the other hand, if
we use too many elements, the inverse problem would be under-determined. Thus, given the level of uncertainties associated to the data, there should be an optimal choice for the discretization level of $\domain{F}$.}\\
This can be seen as a discrepancy principle w.r.t the level of discretization in the domain of $F$.
Therefore, choosing an appropriate level of discretization in the domain of the operator $F$ is crucial for achieving reliable solutions.
One way of obtaining this is by accounting separately for the uncertainty contributions in the discretization levels for the domain and image, and the noise
level from market bid and ask spread when implementing the inverse problem analysis.  
The convergence rate result in Section~\ref{sec:discrete}, is, in
certain sense, a theoretical evidence of this feature, since it
states how distant, in an appropriate sense, a reconstructed local
volatility surface is from the true solution in terms of the three
quantities above mentioned.

This discrete setting gives us sufficient conditions to explore
better the third main point of the present work, the numerical
implementation of different regularization approaches, that allows
us to illustrate the theoretical results presented in Section~\ref{sec:discrete}. This contribution is
presented in Section~\ref{sec:numerical}.

%%%%%%%%%%%%%%%%%%%%%%%%%%%%%%%%%%%%%%%%%%%%%%%%%%%%%%%%%%%%%%%%%%%%%%%%%%%%%%%%%%%%%%%%%%%%%%%%%%%%%%%%%%%%%%%%%%%%%%%%%%%%%%%%%%%%%%%%%
\paragraph{The Setting and the Inverse Problem:}

We are concerned with the problem of local volatility
calibration from European call options. 
Different versions of the calibration problem have been 
addressed by many authors using techniques different 
from those described here. See \cite{achpi, aveSMF,aveICM, aveIJTAF, ABFGKN2000, AFHS97, BBF2000,BouIsak1997, Crepey-2003,dermankanizou,Egger-Engl2005, HKPS-2007,Hof-Kra-2005,JackSuliHow99,jackwerth,lagnado,rlee2001,rlee2005,samperi}.%CLV1998,,stutzer

We recall that, an European call is a contract that gives the right, but not
the obligation, of paying a fixed amount, known as strike price, for
one share of its underlying asset at maturity. To price such contract in 
the context of complete markets, we consider a filtered
probability space $(\Omega,\mathcal{G},\mathbb{F},\mathbb{Q})$,
where $\mathbb{F}$ is a filtration and $\mathbb{Q}$ is the risk
neutral measure. Denote by $S_t = S(t,\omega)$, with $\omega \in
\Omega$ and $t > 0$, the price at time $t$ of the above mentioned
asset. We assume that $S_t$ follows the It\^o's dynamics %%% ADC modifications "below:"
\begin{equation}
dS_t = rS_tdt + \sqrt{\nu_t}S_tdW^{\mathbb{Q}}_t,~ ~ ~t>0,~ ~ ~S(0) = S_0 > 0,
\label{ito1}
\end{equation}
where $W^{\mathbb{Q}}$ is a $\mathbb{Q}$-Brownian motion and $\nu_t$
is the stochastic volatility.

In the present context, the price at time $t$ of an European call
with strike $K>0$ and maturity $T \geq t > 0$, written on the asset
$S_t$, can be represented as 
the discounted expectation of the payoff $\max\{0, S_T - K\}$  in
the risk neutral measure $\mathbb{Q}$, i.e.,
\begin{equation}
C(t,S_t,T,K) = \mathbb{E}^{\mathbb{Q}}\left[\left.\text{e}^{-r(T-t)}\max\left\{0,S_T - K\right\} \right|\mathcal{F}_t\right],
\label{oprice}
\end{equation}
where $r \geq 0$ is the annualized risk free interest rate.

We assume that the volatility term in Equation~(\ref{ito1}) is a
deterministic function of time and asset price:
\begin{equation}
 \sqrt{\nu_t} = \sigma(t,S_t).
\end{equation}
It follows that, when we fix the current time and the
current stock price, by setting $t=0$ and $S(t=0) = S_0$, the option price
$C = C(T,K)$, as a function of its maturity and strike, satisfies Dupire's equation~\cite{Dupre-1994}. 
More precisely, making the change of variables $y = \log(K/S_0)$ and $\tau = T-t$, and defining
$$
u(\tau,y) = C(t + \tau,S_0\text{e}^y) ~~\text{ and }~~ a(\tau,y) = \sigma^2(\tau,S_0\text{e}^y)/2,
$$
the option prices $u(\tau,y)$ satisfy:
\begin{equation}
\left\{
\begin{array}{rcll}
-u_\tau + a(u_{yy} - u_y) - ru_y &=& 0, & 0< \tau\leq T,~y \in I \subset \mathbb{R} \\
\\
u(\tau=0,y) & = &S_0 (\text{e}^y - 1)^+, & y \in I \subset \mathbb{R}\\
\\
\displaystyle\lim_{y\rightarrow -\infty}u(\tau,y) &=& S_0, & 0< \tau\leq T,\\
\\
\displaystyle\lim_{y\rightarrow +\infty}u(\tau,y) &=& 0, & 0< \tau\leq T,
\end{array}
\right.
\label{1}
\end{equation}
where $I \subset\mathbb{R}$ is an interval. 
Note that $I$ may be unbounded. 

Note that, in the present setup (\ref{ito1})-(\ref{1}), local volatility is non-observable and determines the distribution of the underlying asset. Thus, developing reliable calibration methods for this quantity is crucial in a plethora of applications, such as, hedging strategies based on pricing exotic derivatives. 

Denote by $D: = (0,T] \times I$ the domain of the definition of
Equation~\eqref{1}.
In order to address the volatility calibration problem in a general and rigorous manner, we define below the parameter-to-solution
operator:

\begin{equation}
\begin{array}{rcl}
F:\domain{F} \subset \X & \longrightarrow& \Y,\\
a &\longmapsto& F(a) = u(a)-u(a_0)
\end{array}
\label{operator}
\end{equation}
where $\domain{F} := \{a \in a_0 + \X: \underline{a} \leq a \leq
\overline{a}\}$, with $0< \underline{a} < \overline{a}$ constants
and $a_0 \in \X$ given. This operator associates each local
volatility $a \in \domain{F}$ to the difference $u(a)-u(a_0)$, where
$u(a)$ and $u(a_0)$ are the unique solutions for Problem (\ref{1}) with
local volatilities $a$ and $a_0$, respectively. The set $\X$ with
$\varepsilon \in (0,1]$ is a fractional Sobolev space \cite{Ada75}. We also assume that we observe the solutions of Problem~(\ref{1}) in $\Y$.

Note that, originally a solution for Problem~(\ref{1}) is only
locally summable. Thus, the introduction of $u(a_0)$ in the
definition of $F$ is necessary since $u(a) - u(a_0)$ lies in
$W^{1,2}_p(D)$. This is the space of functions with weak derivatives up to
order one in time and up to order two in space, all them in
$L^p(D)$. For results concerning the existence and uniqueness of solutions for
Problem~(\ref{1}) and regularity properties of the operator defined
in (\ref{operator}), see \cite{Crepey-2003, DCZ2010, Egger-Engl2005}
and references therein.

%%%%%%%%%%%%%%%%%%%%%%%%%%%%%%%%%%%%%%%%%%%%%%%%%%%%%%%%%%%%%%%%%%%%%%%%%%%%%%%%%%%%%%%%%%%%%%%%%%%%%%%%%%%%%

As above mentioned, in this article we are concerned with the problem of calibrating the local volatility surface from traded European call option prices. This so-called inverse problem can be stated as: \\
{\em Given a set of traded option prices $\widetilde{u} \in \Y$, find the associated local volatility surface $a^\dagger \in \domain{F}$ satisfying the following operator equation:}
\begin{equation}
 \widetilde{u} = F(a^\dagger) + u(a_0).
\label{eq:operator1}
\end{equation}
%%%%%%%%%%%%%%%%%

One solution for this inverse problem was proposed by Dupire~\cite{Dupre-1994}, where, under no-arbitrage conditions, we
could find local volatility through the formula (in log-moneyness variables):
\begin{equation}
 a(\tau,y) = \displaystyle\frac{ru_y -u_\tau}{u_{yy} - u_y}.
\label{dupform}
\end{equation}
Unfortunately, this formula is not applicable in practice, since it implies in differentiating a noisy data, which is given in a sparse grid.

Note that, when stating Problem~(\ref{eq:operator1}), we have assumed
that the data $\widetilde{u}$ is known with infinite precision in
the whole domain. In order to be more realistic, we consider
that, when measuring the data we have some sources of uncertainties and
noise, which arise specially from the sparsity of data, bid-ask spread,
delays and interpolation.
In addition, it is well known that the forward operator $F$, defined in
(\ref{eq:operator}), is compact. See \cite{DCOZ2010,
DeCezaroScherzerZubelli09, Egger-Engl2005}. It implies that small perturbations in the data may lead to substantial changes in the
reconstructed solution. Therefore, we state the local volatility calibration problem in more mundane setting as follows:\\
{\em Given the set of traded prices $u^\delta$, find its corresponding local volatility surface $a^\dagger \in \domain{F}$ from the operator equation}
\begin{equation}
 u^\delta = \widetilde{u} + E = F(a^\dagger) + u(a_0) + E,
\label{eq:operator}
\end{equation}
where $\widetilde{u}$ represents the unobservable noiseless data and
$E$ is a quantity compiling and quantifying some of the
uncertainties concerned with the model. The set of noiseless prices
are associated to $a^\dagger \in \domain{F}$ through
Problem~(\ref{eq:operator1}).

Denote by $\eta > 0$ an upper bound for the magnitude of $E$. In what follows, we assume that $\eta$ is known {\it a priori} and $E$ satisfies:
\begin{equation}
\|\widetilde{u} - u^\delta\|_{\Y} = \|E\|_{\Y} \leq \eta.
\label{eq:noise-know}
\end{equation}

In Section~\ref{sec:discrete}, we shall express $\eta$ in terms of the following three different quantities:
\begin{enumerate}
 \item The discretization level on the domain of $F$, which we denote by $\gamma_n \geq 0$;
 \item The discretization level on the range of $F$, which we denote by $\rho_m \geq 0$;
 \item The noise level concerning the financial market uncertainties, which we denote by $\delta > 0$, where $n$ and $m$ are, respectively, the number of elements in the discretization of the domain and range of $F$.
\end{enumerate}

\paragraph{Notation, Definition and Assumptions}
In the remaining part of this section we introduce some
important assumptions and tools in order to proceed with the inverse problem analysis under the convex
regularization framework \cite{Bur-Osher-2004, DeCezaroScherzerZubelli09, HKPS-2007, Resm-2005, Resm-Scher-2005}.
\begin{assumption}\label{ass:1}
The regularization functional $f_{a_0}:\domain{f} \subset \X
\longrightarrow [0,\infty]$ is convex, proper and weakly lower
semi-continuous. Its domain $\domain{f}$ contains
$\domain{F}$. \label{ass:domain_f}
\end{assumption}

Another important tool is the Bregman distance with respect to the functional $f_{a_0}$:

\begin{df}\label{df:subdiff}
Let the functional $f$ satisfy Assumption~\ref{ass:domain_f}. For a
given $a \in \domain{f}$, let $\partial f (a) \subset \X^*$ denote the
sub-differential of the functional $f$ at $a$. We denote by
\begin{align*}
{\cal D} (\partial f) =\{ \tilde{a} : \partial f (\tilde{a})\neq
\emptyset\}
\end{align*}
the domain of the sub-differential $\partial f (a)$ \cite{Clarke-1983}. The Bregman
distance with respect to $\zeta \in \partial f(a_1)$ is defined on
$\domain{f} \times {\cal D} (\partial f)$ by
$$
D_{\zeta}(a_2,a_1)=f(a_2)-f(a_1)-\langle\zeta, a_2-a_1\rangle\;.
$$
In particular, if $f$ is differentiable, $\zeta = f'(\cdot)$.
\end{df}

It is also important to define the Fenchel conjugate:
\begin{df}\label{df:Fenchel_duality}
Let the functional $f$ satisfy Assumption~\ref{ass:domain_f}. The
function $f^*: \X^* \to \mathbb{R}$ given by
$$ f^*(x^*):= \sup_{x\in \X}\{\langle x, x^* \rangle -
f(x)\} $$
is called the Fenchel conjugate of $f$ \cite{Rockafellar74}.
\end{df}

Finally, the concept defined below is strongly related to the convergence rates results of Section~\ref{sec:convex}:
\begin{df}\label{def:q-coercivity-bregman}
Let $1\leq q< \infty$ be fixed. The Bregman distance $D_\zeta(\cdot,
\tilde{a})$ of $f: \X \to \mathbb{R} \cup \{+\infty\}$ at
$\tilde{a}\in \mathcal{D}_B(f)$ and $\zeta \in
\partial f$ is said to be $q$-coercive  with constant
$\underline{c}> 0$ if
\begin{align}\label{eq:bregman-coercive}
D_\zeta(a,\tilde{a}) \geq \underline{c} \norm{a - \tilde{a}}^q_{\X}
\mbox{, } \quad \forall a\in \domain{f}.
\end{align}
\end{df}

%-----------------------------------------------------------------------------------------------
%-----------------------------------------------------------------------------------------------
\section{Convex Regularization: A Review of Literature on Local Volatility Calibration}\label{sec:convex}

It is well known, theoretically as well as in practice 
\cite{DeCezaroScherzerZubelli09, Crepey-2003,
Egger-Engl2005,Hof-Kra-2005, HKPS-2007}, that
the local volatility calibration is a highly ill-posed inverse problem. In other
words, small changes and noise in the data may lead to
substantial changes in the reconstructions. Given the importance of good
model selection for the volatility surface, a great amount of effort
had been made in order to tackle the ill-posedness of the
calibration \cite{DeCezaroScherzerZubelli09, Crepey-2003,
Egger-Engl2005,Hof-Kra-2005, HKPS-2007,JackSuliHow99}.

For notational simplicity we formulate this calibration problem in an infinite dimensional framework.
Hereafter, we shall assume again that $D$ is a bounded subdomain of
$\mathbb{R}^2$. Within this framework, we remark that $ \domain{F}
\subset \X\cap L^\infty_{>0}(D) \subset L^1(D)$, where
$L^\infty_{>0}(D)$ is the set of functions that are (essentially)
bounded from below and above by some positive constants.

Successful strategies to stabilize the problem are connected 
to the introduction of
regularization techniques. In particular, many regularization
techniques have been proposed in order to overcome the stability
difficulties, e.g. \cite{DeCezaroScherzerZubelli09, Crepey-2003,
Egger-Engl2005,Hof-Kra-2005, HKPS-2007} and references therein. 
The most common regularization
strategy for volatility calibration in the literature includes
entropy based regularization \cite{AFHS97}, or convex
type regularization  \cite{DeCezaroScherzerZubelli09, Crepey-2003,
Egger-Engl2005,Hof-Kra-2005, HKPS-2007}. 
A unified approach of those
Tikhonov-type regularization was derived in \cite{DeCezaroScherzerZubelli09} by means of the
minimization of the Tikhonov functional
\begin{equation}
\F_{\beta,\ud}(a):= \|F(a)-\ud\|^2_{\Y}+\beta f_{a_0}(a)
\label{eq:reg-functional}
\end{equation}
over $ \domain{F}$. 
In \eqref{eq:reg-functional}, $\beta > 0$ is the regularization
parameter and $f_{a_0}(\cdot)$ is the regularization functional that
is assumed to satisfy the Assumption~\ref{ass:1}.
In particular, we have the following examples:

\begin{ex}[Standard Quadratic Regularization]\label{ex:quadratic}
Here we consider:
\begin{equation}
f_{a_0}(a) = \norm{a-a_0}^2_{L^2(\Omega)}\;,
\label{eq:standard-Tikhonov-functional}
\end{equation}
A particular example that we are interested in is the finite
dimensional version of the Tikhonov-type
functional~\eqref{eq:standard-Tikhonov-functional} that reads as
follows: Let $\{\phi_n\}$ be an orthonormal basis for $L^2
(\Omega)$. Then, consider the minimization of the functional
\begin{equation}\label{eq:Tikhonov-Lq}
\a \longmapsto \F_{\beta,\ud}(\a) :=
\norm{F(\a)-\ud}_{L^2(\Omega)}^2 + \beta \sum_{j=1}^n  |\langle a -
a_0, \phi_j\rangle|^2\,.
 \end{equation}
This Tikhonov functional can be interpreted as a finite dimensional
version of the quadratic regularization. In particular, this is an
interesting penalization functional to consider in the discrete
regularization approach in Section~\ref{sec:discrete} and for
numerical purposes in Section~\ref{sec:numerical}.
\end{ex}

\begin{ex}[Kullback-Leibler Regularization]\label{ex:KL}
Let us now consider:
\begin{equation}
f_{a_0}(a) = \KL{a_0}{\a}\;,
\label{eq:Tikhonov-functional}
\end{equation}
where
\begin{equation*}
 \KL{a_0}{\a} = \int_D \a \log (a_0/ \a) - (a_0 - \a) \, dx\;\mbox{ .}
\end{equation*}
\end{ex}
Note that the Kullback-Leibler distance is the Bregman distance
associated to the Boltzmann-Shannon entropy
\begin{equation}
{\cal G}(\a):=\int_D \a \log (\a) \, dx\;.
\end{equation}

\begin{rm}
 One important motivation for convex regularization follows from the fact
that, in some cases, if the sequence $\{a_k\}$ converges weakly to
$a$ and also satisfies that
$f_{a_0}(a_k)\longrightarrow f_{a_0}(a)$ as $k \rightarrow \infty$
then it convergences strongly. This property is known as the
Radon-Riesz, Kadec-Klee property, or yet H-property
\cite{Zalinescu02}. In particular, for locally uniformly convex and
reflexive Banach spaces and some Besov spaces, such a property is
fulfilled when $f_{a_0} = \norm{\cdot}^p$, for $1 < p < \infty$. See
\cite{Ada75, EkeTem76, Zalinescu02}. Moreover, when we consider
$L^1$ with its weak topology and $f_{a_0}$ as the Boltzmann-Shannon
entropy, such property is satisfied. See \cite{Zalinescu02}.
\label{rmk-Hprop}
\end{rm}

\begin{remark}
The domains of ${\cal G}$,  $\domain{{\cal G}}$,  and of the
subgradient of ${\cal G}$, $\domain{\partial {\cal G}}$, are
$L_{\geq 0}^\infty(D)$ (the set of bounded non-negative functions)
and $L_{> 0}^\infty(D)$, respectively.

The Kullback-Leibler distance, which is the Bregman distance of the
Boltz\-mann-Shannon entropy, is defined on the Bregman domain ${\cal
D}_B({\cal G})$, that is a subset of $L_{>0}^\infty(D)$. Moreover,
the Kullback-Leibler distance is lower semi-continuous with respect
to the $L^1$-norm \cite{ResAnd07}. Based on this property we extend
the Kullback-Leibler distance, to take value $+\infty$ if either $\a
\notin \domain{\cal{G}}$ or $b \notin {\cal D}_B({\cal G})$.

Note that, there are exceptional cases, when the integral
$$
\int_D \a \log (\a/ a_0) - (\a - a_0)\,dx
$$
is actually finite, but $\KL{\a}{a_0} = \infty$. This can be
seen by taking for instance $a \in L_{>0}^1(D)$ which is not in
$L^\infty(D)$ and $a_0=C \a$, where $C$ is a constant. The
reason here, is that $\a$ is not an element of the subgradient of
the Boltzmann-Shannon entropy. This follows directly from the
definition of the domain of the convex functionals and their subgradients.
\end{remark}

%%%%%%%%%%%%%%%%%%%%%%%%%%%%%%%%%%%%%%%%%%%%%%%%%%%%%%%%%%%%%%%%%%%%%%%%%%%%%%%%%%%%%%%%%%%%%%%%%%%%%%%%%%%%%%%%%%%%%%%%%%%%%%%5
\subsection{Stability and Convergence}\label{sec:stability}

In this section we shall present the result on well-posedness,
stability and convergence of minimizers of the Tikhonov functional
$\F_{\beta,\ud}$ defined in \eqref{eq:reg-functional}.

Consider the forward operator $F$ defined in \eqref{eq:operator} and the spaces 
$\X$ and $\Y$ with its weak topologies. It follows from \cite{DeCezaroScherzerZubelli09} 
that for any $\beta >0$ and $\ud \in \Y$, there exists a minimizer of $\F_{\beta,\ud}$. In other words, the regularized problem has a solution.
In addition, we have that such minimizers are stable in the following sense: 
If $(u_k)$ is a sequence converging to $u$ in
$\Y$ with respect to the norm topology, then every sequence $(a_k)$
with
\begin{equation*}
a_k \in \argmin \bigl\{\F_{\beta,u_k}(a): a \in {\cal D}\bigr\}
\end{equation*}
has a subsequence which converges weakly. The limit of every weak
convergent subsequence $(a_{k'})$ of $(a_k)$ is a minimizer $\tilde{a}$ of $\F_{\beta,u}$.
We also have that, when the noise level vanishes, the regularized solutions converge to the true local volatility surface in the following sense:
Let also the map $\beta: (0,\infty) \to (0,\infty)$
satisfy
$$
\beta(\delta)\to 0 \text{ and } \delta^2/\beta(\delta) \to 0\,,
\text{ as } \delta \to 0\;.
$$
Moreover, let the sequence $(\delta_k)$ converge to $0$. Assume that
$(u_k)$ with $u_k := u^{\delta_k}$ is a sequence in $\Y$ satisfying
$\norm{\bar{u} - u_k} \leq \delta_k$. Set $\beta_k :=
\beta(\delta_k)$. Then, every sequence $(a_k)$ of minimizers of
$\F_{\beta_k,u_k}$, has a subsequence $(a_{k'})$ converging weakly.
Each limit $a^\dagger$ of such subsequences is an $f_0$-minimizing
solution of \eqref{eq:operator}, i.e., a solution of \eqref{eq:operator} that is a minimum of $f_{a_0}(\cdot)$.
In other words, when the uncertainty level $\delta$ goes to zero and the parameter of regularization
$\beta = \beta(\delta)$ is properly chosen, then the approximate solutions
converge to the solution of \eqref{eq:operator}. An illustration of this theorem is given in Section~\ref{sec:numerical}.

From the above discussion it is easy to see that the convex Tikhonov-type functional defined in Example~\ref{ex:quadratic} has stable minimizers that converge to the true local volatility surface when the noise level vanishes. However, for the Kullback-Leibler regularization approach of Example~\ref{ex:KL}, it is necessary to make additional assumptions. Indeed, in order to state the convergence analysis for the Kullback-Leibler regularization, it is necessary to consider the weak-to-weak convergence of the forward operator (see Theorem~\ref{th:prop-F} item (ii)) combined with the following inequality \cite{ResAnd07},

\begin{align}
\norm{a- b}^2_{L^2(\Omega)} \leq 2^{-1} KL(a,b) \qquad \forall
a,b\in \mathcal{D}(\mathcal{G})\,.
\label{eq:KL2}
\end{align}
It guarantees the closedness of the graph of the Tikhonov functional \eqref{eq:Tikhonov-functional} with Kullback-Leibler regularization. See \cite[Theorem 27]{DeCezaroScherzerZubelli09}.

%%%%%%%%%%%%%%%%%%%%%%%%%%%%%%%%%%%%%%%%%%%%%%%%%%%%%%%%%%%%%%%%%%%%%%%%%%%%%%%%%%%%%%%%%%%%%%%%%%%%%%%%%%%%%%%%%%%%%%%%%%%%%%%%%%%%%%
\subsection{Convergence Rates}\label{sec:att-sc}

Under the same assumptions of the previous sections, we now recall a convergence rate result for the regularized solutions in terms of the noise level $\delta$. 
Such estimate is based on a prior knowledge about the true 
solution. This is the so-called source condition \cite{EngHanNeu96}. 
By \cite[Lemma~14]{DeCezaroScherzerZubelli09}, for this specific problem, we can guarantee the existence of $w \in \Y^*$ and $r \in \X^*$ such that, the following approximated source condition holds
\begin{align}\label{eq:app-souce-cond}
\zeta^\dag = F'(a^\dag)^* w + r \qquad \mbox{for} \qquad \zeta^\dag
\in
\partial f(a^\dag)\,.
\end{align}

Consider the forward operator $F$ defined in \eqref{eq:operator}. Let
the source condition \eqref{eq:app-souce-cond} be satisfied with
$w^\dag$ and $r$ such that
$$
\big(\underline{c} \|w^\dag\|_{\Y} + \|r\|_{\Y}\big) :=\beta_1 \in [0,1).
$$
Let also the Bregman distance with respect to $f_{a_0}$ be $q$-coercive with constant $\underline{c}\geq e^{-2}$, as in Definition~\ref{def:q-coercivity-bregman}.
Moreover, let $\beta: (0,\infty) \to (0,\infty)$ satisfy $\beta(\delta) \sim  \delta$. 
Then, by \cite[Theorem~12 and Lemma~16]{DeCezaroScherzerZubelli09}, it follows that
\begin{equation*}
    D_{\zeta^\dag}(a_\beta^\delta,a^\dagger) = \mathcal{O}(\delta)\,,
     \quad
    \left \| F(a^\delta_\beta)-u^\delta \right\|_{\Y} = \cal O(\delta)\;,
\end{equation*}
and there exists $c>0$, such that $f(a_\beta^\delta) \leq
f(a^\dagger) + \delta/c$ for every $\delta$ with $\beta(\delta) \leq
\beta_{\text{max}}$.
Thus, the results concerning the existence and stability of regularized solutions added to the convergence and convergence rate results imply the H-property of Remark~\ref{rmk-Hprop}. The latter, in turn, depends on the uncertainty level.

Note that the $q$-coercivity assumption of the Bregman distance implies that
\begin{equation}
 \label{eq:ip:coercive_rate}
\norm{a_\beta^\delta - a^\dag}_{\X}  =
\mathcal{O}((\delta)^{\frac{1}{q}})\;.
\end{equation}
In particular, for the Kullback-Leibler regularization we have \cite{DeCezaroScherzerZubelli09} (according to the choice of $f_{a_0}$):

\begin{align}\label{eq:L1convergenvce}
\norm{\a^\delta_\beta - \a^\dag}_{L^2(D)} =
\mathcal{O}(\sqrt{\delta})\qquad \mbox{or} \qquad
\norm{\a^\delta_\beta - \a^\dag}_{L^1(D)} =
\mathcal{O}(\sqrt{\delta}) \,.
\end{align}

In the next section we address the regularization of the volatility calibration problem 
in a discrete setting, providing convergence and convergence rates results in terms of the noise level and the discretization levels in domain and range of the forward operator.

%-----------------------------------------------------------------------------------------------
\section{The Discrete Setting}\label{sec:discrete}
In order to obtain more practical error estimates for the numerical solution of the
Inverse Problem \eqref{eq:operator}, we analyze the Tikhonov-type regularization
presented in Section~\ref{sec:convex} in a discrete setting. The discreteness is modeled by considering fixed
nested sequences of finite-dimensional subspaces of $\X$ and $\Y$
respectively, $\{U_n\}^\infty_{n=0}$ and $\{V_m\}^\infty_{m=0}$ such
that $\overline{\cup^\infty_{n = 0}U_n} = \X$ and
$\overline{\cup^\infty_{m = 0}V_m} = \Y$. Now, define the sets
$\mathcal{D}_n = U_n \cap \mathcal{D}(F)$ for $n \in \mathbb{N}$ and
assume that they are nonempty.

We can now define the operators $F_m : \mathcal{D}(F)\subset \X
\longrightarrow V_m$ approximating $F$ and a discrete version of the
parameter $a \in \mathcal{D}(F)$, that we call $a_n \in
\mathcal{D}_n$, satisfying
\begin{equation}
\|F(a) - F_m(a)\|_{\Y} \leq \rho_m, \quad \norm{a_n - a}_{\X} \leq
\gamma_n \,, \label{p3}
\end{equation}
with  $\lim_{m\rightarrow\infty}\rho_m = 0$  and $ \lim_{n
\rightarrow \infty} \gamma_n = 0$ for all $ a \in \mathcal{D}(F)$. 
In the numerical experiments of Section~\ref{sec:numerical}, the spaces $U_n$ are generated by finite sets  of
finite element functions and the operators $F_m$ are  suitable finite difference operators associated to the Problem (\ref{1}).

In this setting the discrete version of the Tikhonov-type
calibration inverse problem takes the form: 

\noindent Find a solution for the
minimization problem of the following Tikhonov functional
\begin{equation}
\F_{\beta, \ud}^{m,n}:=\|F_m(a) - u^\delta\|^2_{\Y} + \beta
f(a)\quad \mbox{subject to a }\, \in
\mathcal{D}_n, 
\label{eq:discrete-functional}
\end{equation}  where $u^\delta$ is a prices surface satisfying
Equation~\eqref{eq:noise-know}.

As in Section~\ref{sec:convex}, we use convex regularization
techniques to ensure the existence, stability of solution and
convergence rates of the approximate solution sequence of the inverse problem.

We remark that, in the discrete setting we account separately for the uncertainties
originated by the discretization of the operator $F$ and its domain of definition, and the noise introduced by the market noise such as bid and ask spread.

We start with the following auxiliary lemma:
\begin{lemma}
\begin{itemize}
\item[i)] For all $m \in \mathbb{N}$, the finite difference operator
$F_m$ is sequentially weakly closed.
\item[ii)] For every $M>0, \beta > 0$ and $m, n \in \mathbb{N}$,
the sets $\{ a \in \mathcal{D}(F)\,:\, \F_{\beta, \ud}(a) \leq M \}$
and  $\{ a \in \mathcal{D}_n(F)\, :\, \F_{\beta, \ud}^{m,n}(a) \leq
M \}$ are  sequentially weakly compact.
\end{itemize}
\label{lema:preparation}
\end{lemma}

\begin{proof}
Item i) follows from the weak sequential continuity of $F$ (see
\cite[Chapter 1]{DCZ2010}) and definition of the finite different
methods \cite{Wilmott95}. To prove Item ii), we recall that, by
definition, $\mathcal{D}(F)$ and $\mathcal{D}_n(F)$ are convex and
closed. Hence, weakly closed. Therefore, coercivity and weak
semi-continuity of $f_{a_0}$ and Item i) conclude the proof.
\end{proof}

We are now ready to present the analysis of the \textit{a
priori} choice of the regularization parameter $\beta$.

%---------------------------------------------------------------------------------------------
\subsection{\textit{A Priori} Choice for the Regularization Parameter}

In this subsection, we state the well-posedness and perform a convergence analysis
of the discrete version of the calibration problem using the
Tikhonov functional of Equation~\eqref{eq:discrete-functional}.

\begin{pr}
Let $m,n \in \mathbb{N}$ and $\beta, \delta > 0$ be fixed. Moreover,
let \eqref{eq:noise-know} be satisfied. Then, for every $u^\delta
\in \Y$, there exists at least one minimizer of
\eqref{eq:discrete-functional}.

Furthermore, the minimizers of \eqref{eq:discrete-functional} 
are stable with respect to the data $u^\delta$ in the sense of
Section~\ref{sec:stability}
\label{pr:existence}
\end{pr}

The proof is straightforward by the standard arguments presented in
\cite[Theorem 3.23]{SchGraGroHalLen08} and \cite[Proposition
2.3]{PRS2010} jointly with Lemma~\ref{lema:preparation} and
the assumptions on $f_{a_0}$. 

We prove now the auxiliary lemma:

%%%%%%%%%%%%%%%%%%%%%%%%%%%%%%% LEMMA
\begin{lemma}
Assume that an $f_{a_0}$-minimizing solution $a^\dag$ belongs to the
interior of $\mathcal{D}(F)$. 
Moreover, let $\rho > 0$ be small enough
so that the closed ball $B_{\rho}(a^\dag)$ is in the interior of $\mathcal{D}(F)$.
Assume that $a_n \in \mathcal{D}_n(F)\cap B_{\rho}(a^\dag)$, for $n$ sufficiently large, satisfying \eqref{p3}.

Then, $f_{a_0}(a_n) \longrightarrow f_{a_0}(a^\dag)$ as $n \rightarrow \infty$.
\label{lamma:aux}
\end{lemma}

%%%%%%%%%%%%%%%%%%%%%%%%%%%%%%% PROOF
\begin{proof}
Since $\mathcal{D}(F)$ is convex and closed, the convexity and
weak lower semi-continuity of $f_{a_0}$ implies that $f_{a_0}$ is
lower semi-continuous \cite[Corolary 2.2]{EkeTem76}. 
Assumption~\ref{p3} implies that $a_n \rightarrow a$ as $n \rightarrow\infty$.  
Now, \cite[Corolary 2.5]{EkeTem76} implies that $f_{a_0}$
is continuous in the interior of $\mathcal{D}(F)$.
\end{proof}
%%%%%%%%%%%%%%%%%%%%%%%%%%%%% END PROOF

The next theorem ensures that the minimizing solutions of
Equation~\eqref{eq:discrete-functional} are also regularized solutions of
Equation~\eqref{eq:operator}, given a suitable choice for the regularization
parameter when $m,n \rightarrow \infty$. We remark that in
\cite{PRS2010} a similar result was obtained. However they need to
assume stronger conditions. One of the main differences between our
proof and that in \cite{PRS2010} is the validity of the tangential
cone condition in our case. Such condition is reviewed in Theorem~\ref{pr:cone-cond} of Appendix~\ref{sec:2}.
%
%%%%%%%%%%%%%%%%%%%%%%%%%%%%% THEOREM
\begin{theo}
Let \eqref{eq:noise-know}, \eqref{p3} and the assumptions on
Lemma~\ref{lamma:aux} be satisfied. Moreover, let $\beta =
\beta(\rho_m,\gamma_n,\delta)$ be such that
$$\beta \rightarrow 0\,, \qquad  \frac{(\max(\rho_m, \gamma_n,
\delta))^2}{\beta}\rightarrow 0\,,\qquad \mbox{as} \,\, \delta
\rightarrow 0\,\, \mbox{and }\,\, m, n\rightarrow \infty\,.$$

Then, every sequence $\{a_k\}_{k\in \N}$, with $a_k: =
a_{m_k,n_k}^{\beta_k,\delta_k}$ and $\beta_k: =
\beta(\rho_{m_k},\gamma_{n_k}, \delta_k)$ such that $\delta_k
\rightarrow 0$, $m_k\rightarrow \infty$, $n_k \rightarrow \infty$ as
$k\rightarrow \infty$ and $a_k$ is a correspondent minimizer of
\eqref{eq:discrete-functional}, has a weakly convergent subsequence
$\{a_l\}_{l\in \N}$ to an $f_{a_0}$-minimizing solution $a^\dag$ of
\eqref{eq:operator} and for which $\{f_{a_0}(a_l)\}_{l\in \N}$
converges to $f_{a_0}(a^\dag)$. Moreover, if $a^\dag$ is the unique
solution of \eqref{eq:operator}, then the entire sequence $\{a_k\}$
 converges weakly to $a^\dag$.
\label{th:convergence1}
\end{theo}

%%%%%%%%%%%%%%%%%%%%%%%%%%%%% END THEOREM
%%%%%%%%%%%%%%%%%%%%%%%%%%%%% PROOF
\begin{proof}(Sketch)
We make use of standard arguments in the minimization
theory of nonlinear Tikhonov regularization \cite{SchGraGroHalLen08}.

Let $a_{m,n}^{\beta, \delta}$ be a minimizer of
\eqref{eq:discrete-functional} (the existence is guaranteed by
Proposition~\ref{pr:existence}). From the definition of $a_{m,n}^{\beta,
\delta}$, it follows that
\begin{equation}
\begin{array}{rcl}
\norm{F_m(a_{m,n}^{\beta, \delta}) - \ud}^2 + \beta
f_{a_0}(a_{m,n}^{\beta, \delta}) & \leq & \norm{F_m(a_{n}) - \ud}^2 +
\beta f_{a_0}(a_{n})\\
& \leq &\left(\rho_m + \norm{F(a_n) - u} + \delta \right)^2 + \beta
f_{a_0}(a_{n})\,.\nonumber
\end{array}
\label{eq:10.0}
\end{equation}

From the continuity of $F$, the definition of $a_n$ and the tangential cone
condition (see Theorems~\ref{th:prop-F} and \ref{pr:cone-cond} in
the Appendix~\ref{sec:2}), we have
\begin{align}
\norm{F_m(a_{m,n}^{\beta, \delta}) - \ud}^2  + \beta
f_{a_0}(a_{m,n}^{\beta, \delta})  \leq C (\max(\rho_m, \gamma_n,
\delta))^2 + \beta f_{a_0}(a_{n})\,.
\label{eq:10.1}
\end{align}
Therefore,
\begin{align}
f_{a_0}(a_{m,n}^{\beta, \delta}) \leq  C \frac{(\max(\rho_m,
\gamma_n, \delta))^2}{\beta} +  f_{a_0}(a_{n})\,.
\label{eq:10.2}
\end{align}
Lemma~\ref{lamma:aux} and the assumptions on $\beta$ guarantee that
\begin{align}
\displaystyle\limsup_{\stackrel{\delta \rightarrow 0}{m,n\rightarrow \infty}} f_{a_0}(a_{m,n}^{\beta, \delta}) \leq f_{a_0}(a^\dag)\quad
\mbox{and} \quad \displaystyle\lim_{\stackrel{\delta \rightarrow 0}{m,n\rightarrow \infty}} \norm{F_m(a_{m,n}^{\beta, \delta}) - u} =
0\,.
\label{eq:10.3}
\end{align}
Let now $\beta_k: = \beta(\rho_{m_k}, \gamma_{n_k},\delta_k )$ and
$a_k: = a_{m_k,n_k}^{\beta_k, \delta_k}$. Since $\{\F_{\alpha_k,
u^{\delta_k}}^{m_k,n_k}(a_k)\}$ is bounded, the coercivity of
$f_{a_0}$ implies that $\{\norm{a_k}\}$ is bounded. From
Lemma~\ref{lema:preparation}, there exists a subsequence
$\{a_l\}_{l\in \N}$ weakly convergent to some $\tilde{a} \in
\mathcal{D}_n(F)$. Due to the weak lower semi-continuity of
$f_{a_0}$ and \eqref{eq:10.3}, we get that
\begin{align*}
f_{a_0}(\tilde{a}) \leq \liminf_{l \rightarrow \infty}  f_{a_0}(a_l)
\leq  \limsup_{l \rightarrow \infty} f_{a_0}(a_l) \leq
f_{a_0}(a^\dag)\,.
\end{align*}
Continuity of $F$ implies that  $F(\tilde{a}) = u$. Therefore, 
$\tilde{a}$ is an $f_{a_0}$-minimizing solution of
\eqref{eq:operator}.
\end{proof}
%%%%%%%%%%%%%%%%%%%%%%%%%%%%% END PROOF

%--------------------------------------------------------------------------------------
\subsection{Convergence Rates}

We now state results concerning the convergence rates for the regularized
 solutions in the discrete setting. As in the
continuous version presented in Section~\ref{sec:convex},
we need some \textit{a priori} knowledge about the true
solution, which leads us to impose some source condition. Thus, we make use of
the same source condition used in the continuous case, which is
given by \eqref{eq:app-souce-cond}. Therefore, we are ready to prove the following theorem:
%%%%%%%%%%%%%%%%%%%%%%%%%%%%% THEOREM
\begin{theo}
Suppose that $f_{a_0}$ satisfies the
Definition~\ref{def:q-coercivity-bregman}. Let \eqref{eq:noise-know}
and the source  condition \eqref{eq:app-souce-cond} be satisfied. If
$\beta \sim \max(\rho_m,\gamma_n, \delta)$ and $\beta \beta_2 < 1$,
then
\begin{align}\label{eq:conv-rates1}
D_{\zeta^\dag}(a_{m,n}^{\beta, \delta}, a^\dag) =
\mathcal{O}(\max(\rho_m,\gamma_n, \delta))\,.
\end{align}
\label{th:convergence-rates1}
\end{theo}
%%%%%%%%%%%%%%%%%%%%%%%%%%%%% END THEOREM
%%%%%%%%%%%%%%%%%%%%%%%%%%%%% PROOF
\begin{proof}
Let $\{a_n\}_{n\in \N}$ be as in Theorem~\ref{th:convergence1}. From
the definition of $a_{m,n}^{\beta, \delta}$ and estimate analogous
to \eqref{eq:10.0}, it follows that
\begin{align}\label{eq:2.6-1}
\norm{F_m(a_{m,n}^{\beta, \delta}) - \ud}^2 + \beta
f_{a_0}(a_{m,n}^{\beta, \delta}) \leq C (\max(\rho_m, \gamma_n,
\delta))^2 + \beta f_{a_0}(a_n)\,.
\end{align}
Now, using \eqref{eq:2.6-1} and the definition of the Bregman
distance we have that
\begin{align}\label{eq:2.6-2}
\norm{F_m(a_{m,n}^{\beta, \delta}) - \ud}^2 & + \beta
D_{\zeta^n}(a_{m,n}^{\beta, \delta}, a_n) \nonumber\\
& =
\norm{F_m(a_{m,n}^{\beta, \delta}) - \ud}^2 + \beta \left( f_{a_0}(a_{m,n}^{\beta, \delta}) - f_{a_0}(a_n) - \langle \zeta^n, a_{m,n}^{\beta, \delta} - a_n \rangle \right)\nonumber\qquad\\
& =
\norm{F_m(a_{m,n}^{\beta, \delta}) - \ud}^2 + \beta f_{a_0}(a_{m,n}^{\beta, \delta}) - \beta \left(f_{a_0}(a_n) + \langle \zeta^k, a_{m,n}^{\beta, \delta} - a_n \rangle \right)\nonumber\\
& \leq
C (\max{\rho_m, \gamma_n, \delta})^2 -  \beta  \langle \zeta^k, a_{m,n}^{\beta, \delta} - a_n \rangle \\
& =
C (\max{\rho_m, \gamma_n, \delta})^2 -  \beta  \langle \zeta^k, a_{m,n}^{\beta, \delta} - a^\dag \rangle  - \beta \langle \zeta^k, a^\dag  - a_n\rangle\nonumber\\
& = C (\max(\rho_m, \gamma_n, \delta))^2
% (1)
\underbrace{-  \beta  \langle \zeta^k- \zeta^\dag, a_{m,n}^{\beta,
\delta} - a^\dag \rangle}_{(A)}
% (2)
\underbrace{ -  \beta  \langle \zeta^\dag, a_{m,n}^{\beta, \delta} - a^\dag \rangle}_{(B)}\nonumber\\
 % (3)
& \qquad \quad \qquad \underbrace{- \beta \langle \zeta^k -
\zeta^\dag, a^\dag  - a_n\rangle}_{(C)}
% (4)
\underbrace{- \beta \langle \zeta^\dag, a^\dag  -
a_n\rangle}_{(D)}\nonumber\,.
\end{align}

To conclude the proof, it remains to estimate the terms $(A)$, $(B)$, $(C)$ and $(D)$ 
in Equation~\eqref{eq:2.6-2} in terms of Bregman distance. At this point we
will strongly use the existence of the source condition
\eqref{eq:app-souce-cond}.

From \cite[Lemma 1.3.9 iv)]{ButIusem2000}, it follows that
\begin{align}\label{eq:(1)}
(A)\,\, = \, -\beta \left[D_{\zeta^\dag}( a_{m,n}^{\beta, \delta}
,a^\dag) + D_{\zeta^n}(a^\dag, a_n) - D_{\zeta^n}( a_{m,n}^{\beta,
\delta}, a_n)\right]\,.
\end{align}
From the convergence rates of Section~\ref{sec:att-sc}, it follows that
\begin{align}\label{eq:(2)}
(B)\,\, & \leq \beta |\langle \zeta^\dag, a_{m,n}^{\beta, \delta} -
a^\dag \rangle| \leq
\beta \beta_1 D_{\zeta^\dag}(a_{m,n}^{\beta, \delta}, a^\dag) +\beta \beta_2 \norm{F(a_{m,n}^{\beta, \delta}) -F(a^\dag)}\\
& \leq \beta \beta_1 D_{\zeta^\dag}(a_{m,n}^{\beta, \delta}, a^\dag)
+\beta \beta_2 \left(\norm{F(a_{m,n}^{\beta, \delta}) -
F_m(a_{m,n}^{\beta, \delta})} +\norm{F_m(a_{m,n}^{\beta, \delta})
-\ud} + \norm{\ud-F(a^\dag)}\right)\,.\nonumber
\end{align}
Using Cauchy-Schwartz inequality, and the result in Reference \cite[Proposition 1.1.7
ii)]{ButIusem2000}, it follows that
\begin{align}\label{eq:(3)}
(C)\,\, \leq \beta \norm{\zeta^n - \zeta^\dag}\norm{a_n - a^\dag}
\leq \beta C \norm{a_n - a^\dag}\,.
\end{align}
Finally, from the source representation
\eqref{eq:app-souce-cond} and the Lipschitz continuity of $F'(a^\dag)$, 
we have that
\begin{align}\label{eq:(4)}
(D)\,\, & = \beta \left(\langle r - \zeta^\dag, a^\dag  - a_n\rangle
+ \langle r, a_n - a^\dag \rangle\right) =
\beta \left(\langle F'(a^\dag)^* w, a^\dag  - a_n\rangle + \langle r, a_n - a^\dag \rangle\right)\nonumber\\
&  \leq \beta \left(\norm{w}\norm{F'(a^\dag)( a^\dag  - a_n)} +
\norm{ r}\norm{a_n - a^\dag} \right) \leq \beta \left(C\norm{w}+
\norm{ r}\right)\norm{a_n - a^\dag} \,.
\end{align}

Putting together \eqref{eq:2.6-2}, and the estimates $(A)$, $(B)$, $(C)$
and $(D)$, the assumptions on the parameters and the fact that
$D_{\zeta}(a,\tilde{a}) \geq 0$, we get the conclusion.
\end{proof}
%%%%%%%%%%%%%%%%%%%%%%%%%%%%% END PROOF

It is worth noticing that under different assumptions from the ones we have 
in Theorem~\ref{th:convergence-rates1}, a similar result was
proved in \cite{PRS2010}. However, we have
weaker assumptions than those in \cite{PRS2010}. Moreover, all the
assumptions are verified for the calibration problem under consideration.

\begin{remark}
The estimate~\eqref{eq:conv-rates1} accounts separately for the
contributions of different sources of uncertainties. It improves the
error estimates when performing numerical tests, making more
reliable the solutions obtained by these techniques.

In particular, the estimate~\eqref{eq:conv-rates1} implies an upper
bound for the discretization level of $\mathcal{D}(F)$ w.r.t. the 
quality of the observation and the quantity of data. In other words, a 
discrepancy principle is obtained (as a function of the
discretization level in $\mathcal{D}(F)$) at the first iteration for which 
$\rho_m \leq \max\{\gamma_n, \delta\}$. This discrepancy principle will be verified numerically in Section~\ref{sec:numerical}. A rigorous proof is the subject of \cite{ADZ2013}.
\end{remark}

%-----------------------------------------------------------------------------------------------
\section{Numerical Results}\label{sec:numerical}
We numerically solve the minimization problem in Equation~(\ref{eq:discrete-functional}) connected to Tikhonov regularization of the inverse problem of local volatility calibration. It is done by a gradient method. See \cite{vvlathesis,vvla2}. 
The direct problem of Equation~(\ref{1}), as well as, its related adjoint, arising in the calculation of the Tikhonov's functional gradient, are numerically solved by a Crank-Nicholson scheme. See \cite{vvlathesis,vvla2,DCZ2010,Wilmott95}.
Some of the theoretical statements of past sections, such as convergence analysis and rates as well as the proposed discrepancy principle in the choice of the discretization level, are now illustrated. 
We also present some reconstructions of the local volatility surface from traded option data.

%------------------------------------------------------------------------------------------------
\subsection{Synthetic Examples}\label{sec:synthetic}

We generate the synthetic data of option prices as follows:
We first choose a very fine mesh and calculate the option prices by the numerical solution of Problem~(\ref{1}) with the local volatility given in Equation~(\ref{vol}). We apply to the resulting prices a zero-mean additive Gaussian white noise with variance  $1\%$ of the maximum value of such option prices. Then, we choose the resulting perturbed data in a coarser grid contained in the first one. This is done in order to avoid the so-called inverse crime \cite{somersalo}.

The local volatility surface used in the present set of numerical experiments is defined by:
\begin{equation}
\sigma(\tau,y) = \left\{
\begin{array}{ll}
\displaystyle\frac{2}{5}-\frac{4}{25}\text{e}^{-\tau/2}\cos\left(\displaystyle\frac{4\pi y}{5}\right),& \text{ if } -2/5 \leq y \leq 2/5\\
\\
2/5,& \text{ otherwise.}
\end{array} \right.
\label{vol}
\end{equation}
Recall that the diffusion parameter of Problem~\ref{1} is $a = \sigma^2/2$.
Note also that, Problem~(\ref{1}) is numerically solved in the domain $D = [0,1]\times[-5,5]$.

We choose the standard two-dimensional linear basis for the domain of the forward operator, since it is smooth. Such smoothness is a feature favored by practitioners \cite{DupPVT}. 
In order to find its appropriate level of discretization, we also assume that the uncertainty levels $\delta$ and $\rho_m$ are given and proceed as follows:
\begin{enumerate}
 \item Chose different meshes that correspond to different levels of discretization of $\mathcal{D}(F)$ with the spline basis. We require that the coarser mesh is contained in the finner one.
 \item For each mesh, solve the inverse problem with the parameter of regularization $\beta$ chosen by the related Morozov's discrepancy principle. See \cite{vvla2,Mor84}.
 \item Choose the coarser grid which satisfies the discrepancy criterion:
\begin{equation}
\tau_1 \max(\delta,\rho_m) \leq \|F_m(a^{\beta,\delta}_{m,n}) - u^\delta\| \leq \tau_2 \max(\delta,\rho_m),
\label{discrepancy} 
\end{equation}
where, $m$ and $\delta$ are fixed and $\tau_1$ and $\tau_2$ are given constants.
\end{enumerate}
The intuitive justification of the above procedure is that:\\
{\em If the level of discretization is too small, the reconstructed local volatility does not capture the variability of the original one. On the other hand, if we choose a level of discretization finer than necessary, we start to reproduce noise.}\\
Such a conclusion can be partially motivated by the convergence rates of Theorem~\ref{th:convergence-rates1}.

For the present examples, we chose $\tau_1 = 1.05$ and $\tau_2 = 1.5$. Furthermore, we consider that $\rho_m = 0$, since we use a fine mesh in the numerical solution of Problem~(\ref{1}). Figure~\ref{test1} illustrates the above mentioned discrepancy principle to find the appropriate level of discretization for $\domain{F}$.

The data was generated using the step sizes, $\Delta t = 0.0025$ and $\Delta y = 0.01$. After adding noise, we collected the data in a mesh with step sizes $\Delta t = 0.02$ and $\Delta y = 0.1$. 
The direct problem in Equation~(\ref{1}) and its adjoint were numerically solved with $\Delta t = 0.02$ and $\Delta y = 0.1$. 
The reconstruction of Figures~\ref{test1} and \ref{test2} were calculated with meshes in $\domain{F}$ with the following step sizes: $$
\Delta t = 0.1, 0.08, 0.07, 0.06, 0.05, 0.04, 0.03, 0.02, 0.01, 0.0075, 0.005, 0.0025
$$
and
$$
\Delta y = 0.1, 0.08, 0.07, 0.06, 0.05, 0.04, 0.03, 0.02, 0.01, 0.0075, 0.005, 0.0025.
$$

\begin{figure}[ht]
\centering
\includegraphics[width=0.5\textwidth]{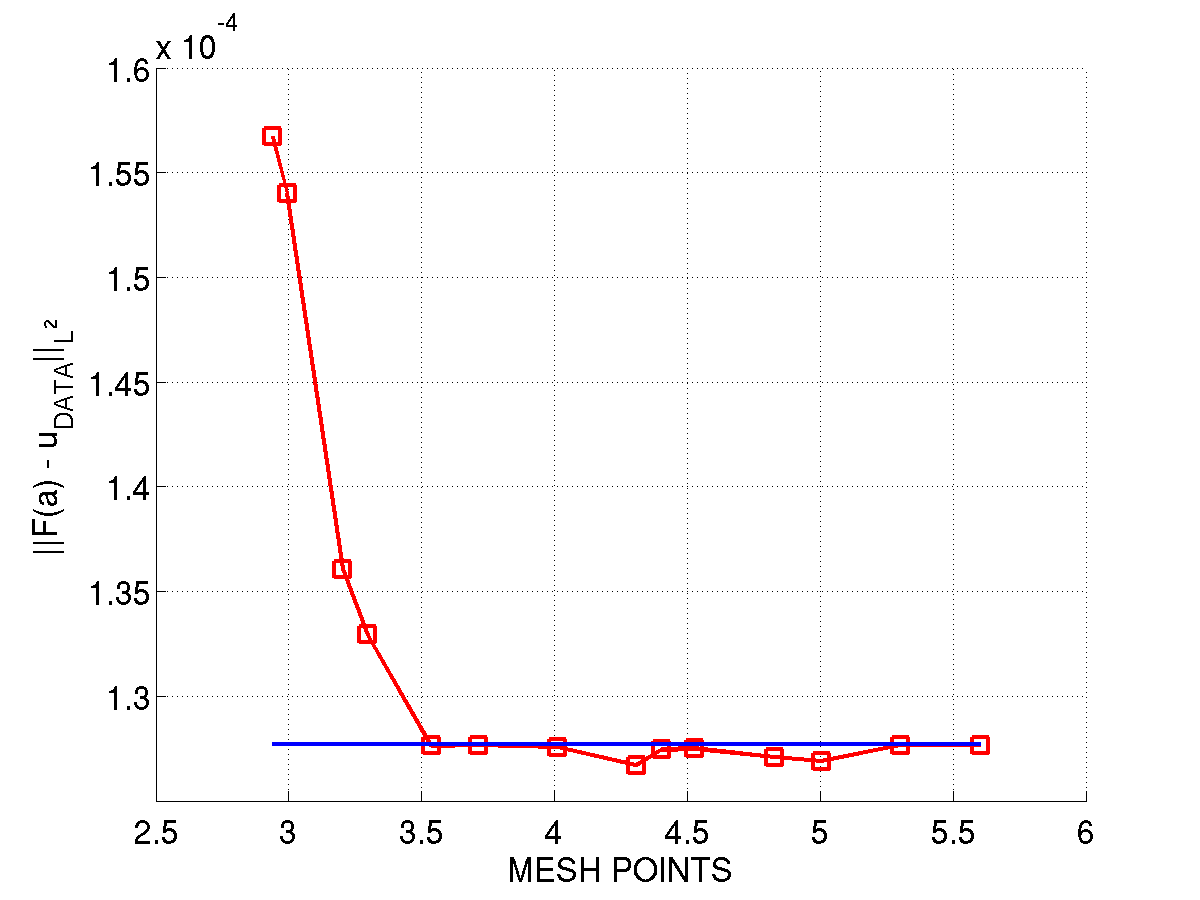}
\caption{Evolution of the residual as a function of the number of mesh points. A discrepancy principle was used to find the appropriate level of discretization of $\mathcal{D}(F)$. In the presence of noise, its minimum is attained for a coarser mesh.The horizontal line corresponds to $\tau_1\delta$.}
\label{test1}
\noindent %\hrulefill
\end{figure}

\begin{figure}[ht]
\centering
\includegraphics[width=0.5\textwidth]{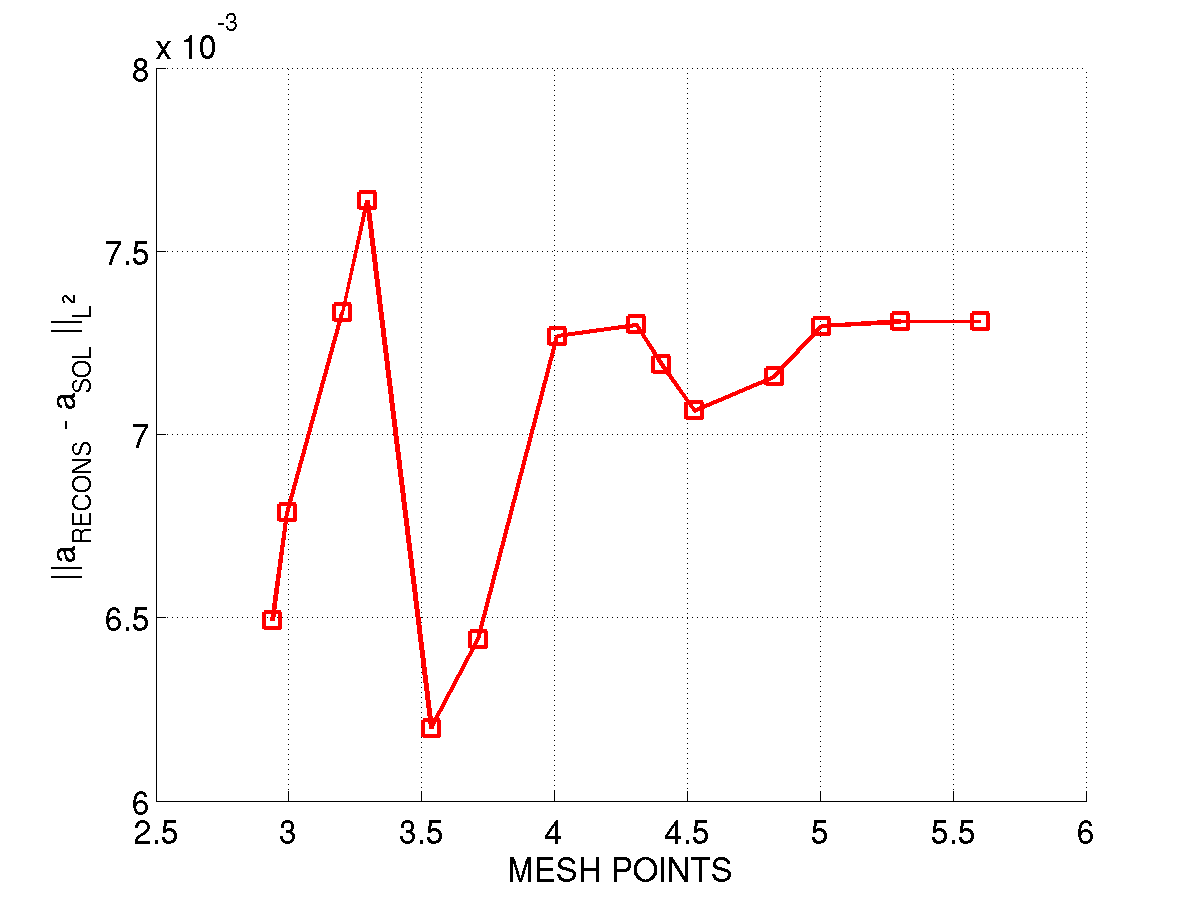}
\caption{Evolution of the $L^2$-error. In the presence of noise, its minimum is attained for a coarser mesh.}
\label{test2}
\noindent %\hrulefill
\end{figure}

\begin{figure}[ht]
\centering
\includegraphics[width=0.33\textwidth]{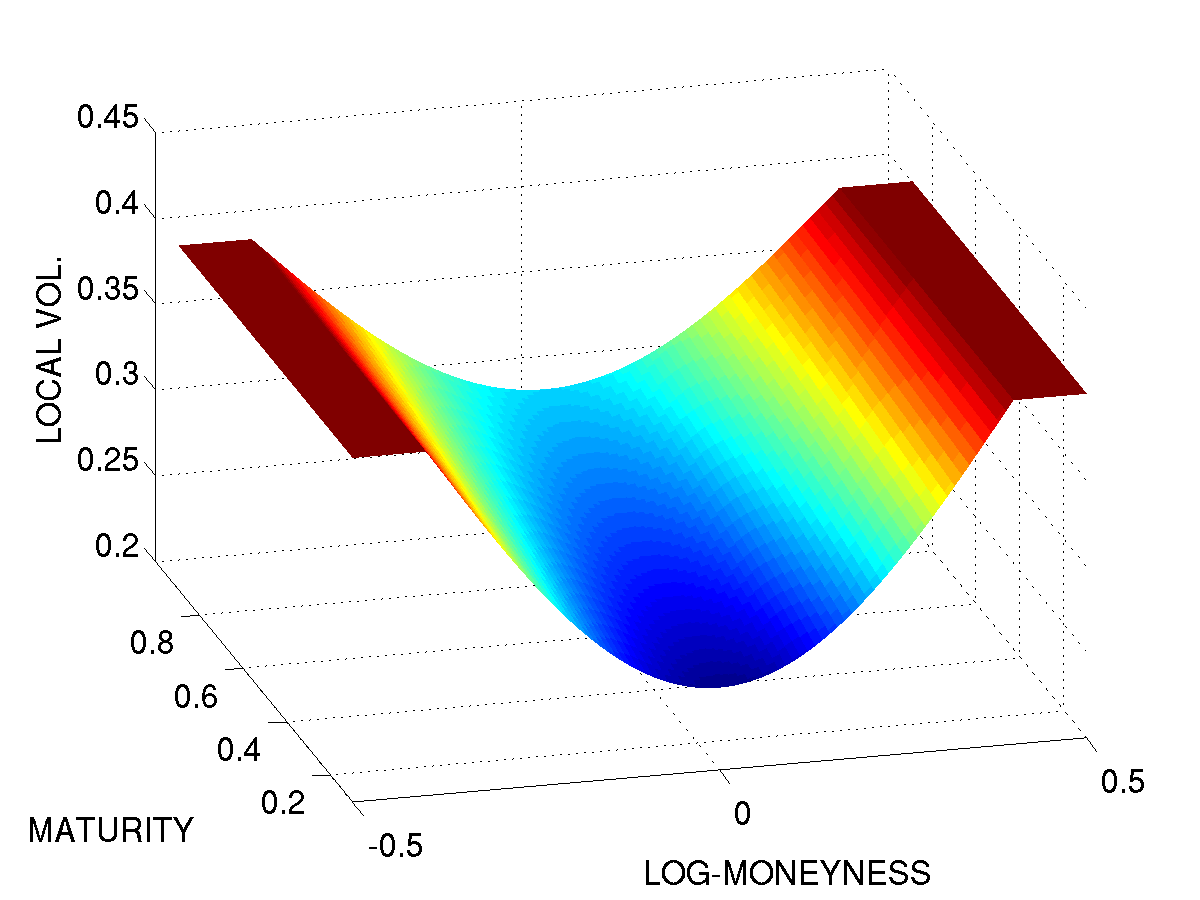}\hfill 
\includegraphics[width=0.33\textwidth]{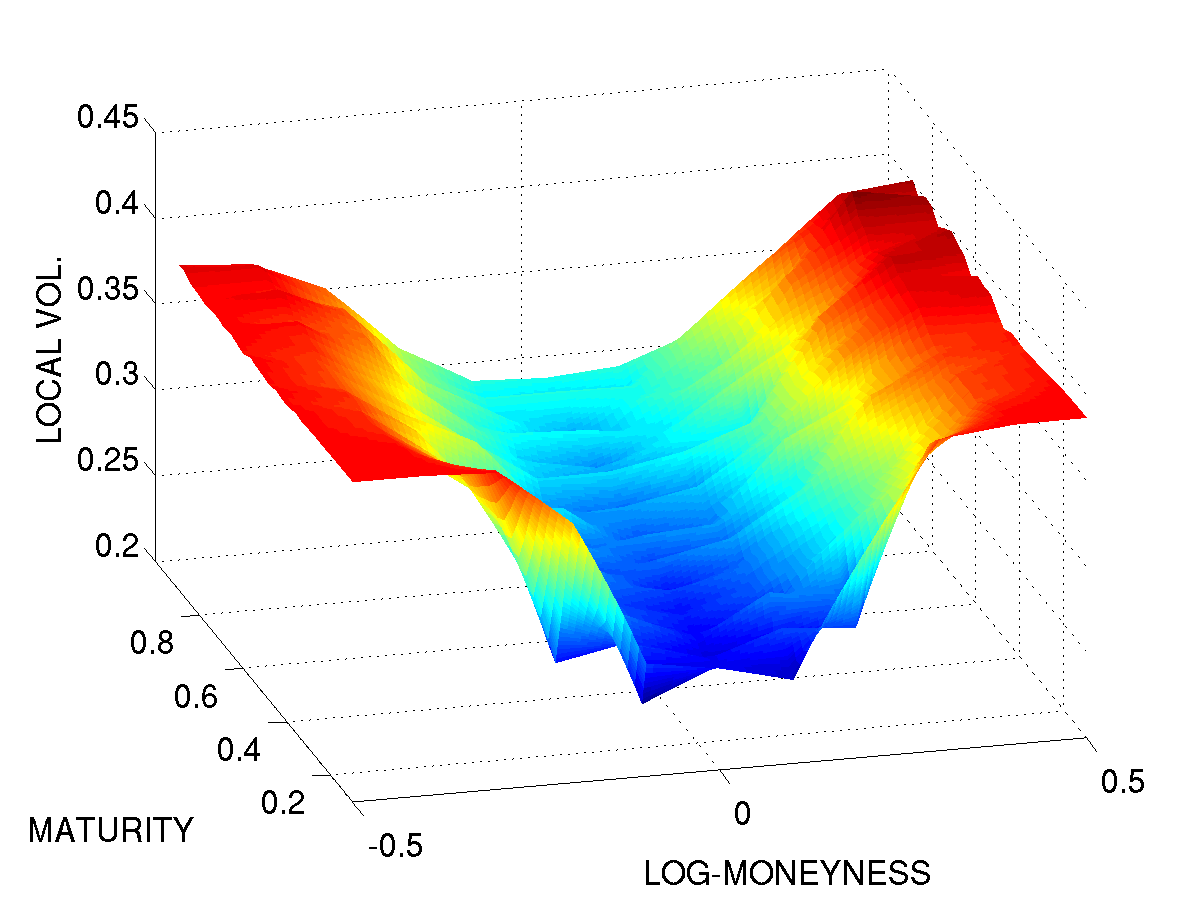}\hfill
\includegraphics[width=0.33\textwidth]{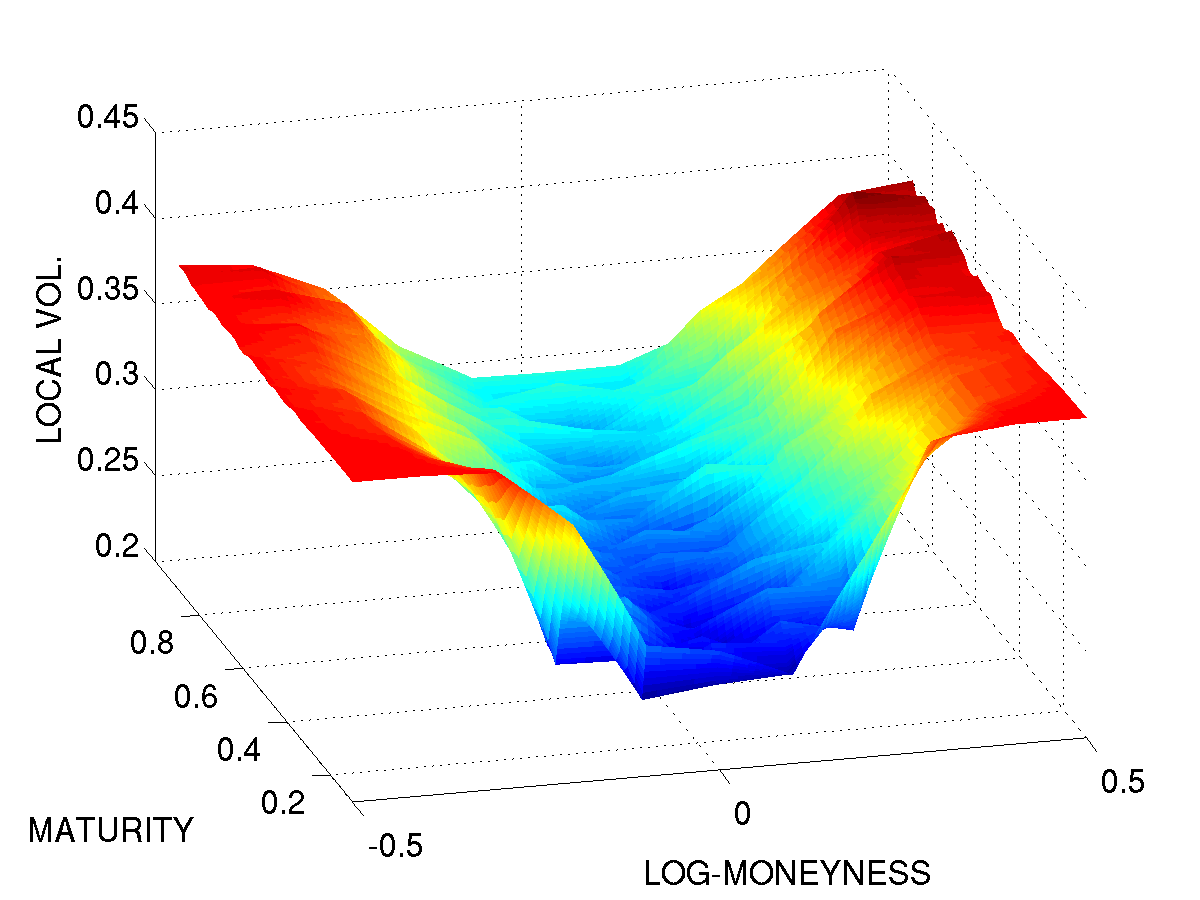}
\caption{Left: original surface. Center and right: reconstructions corresponding to the first and second points satisfying the discrepancy principle of Figure~\ref{test1}, respectively.}
\label{test3}
\noindent %\hrulefill
\end{figure}

\begin{figure}[ht]
\centering
\includegraphics[width=0.33\textwidth]{original}\hfill 
\includegraphics[width=0.33\textwidth]{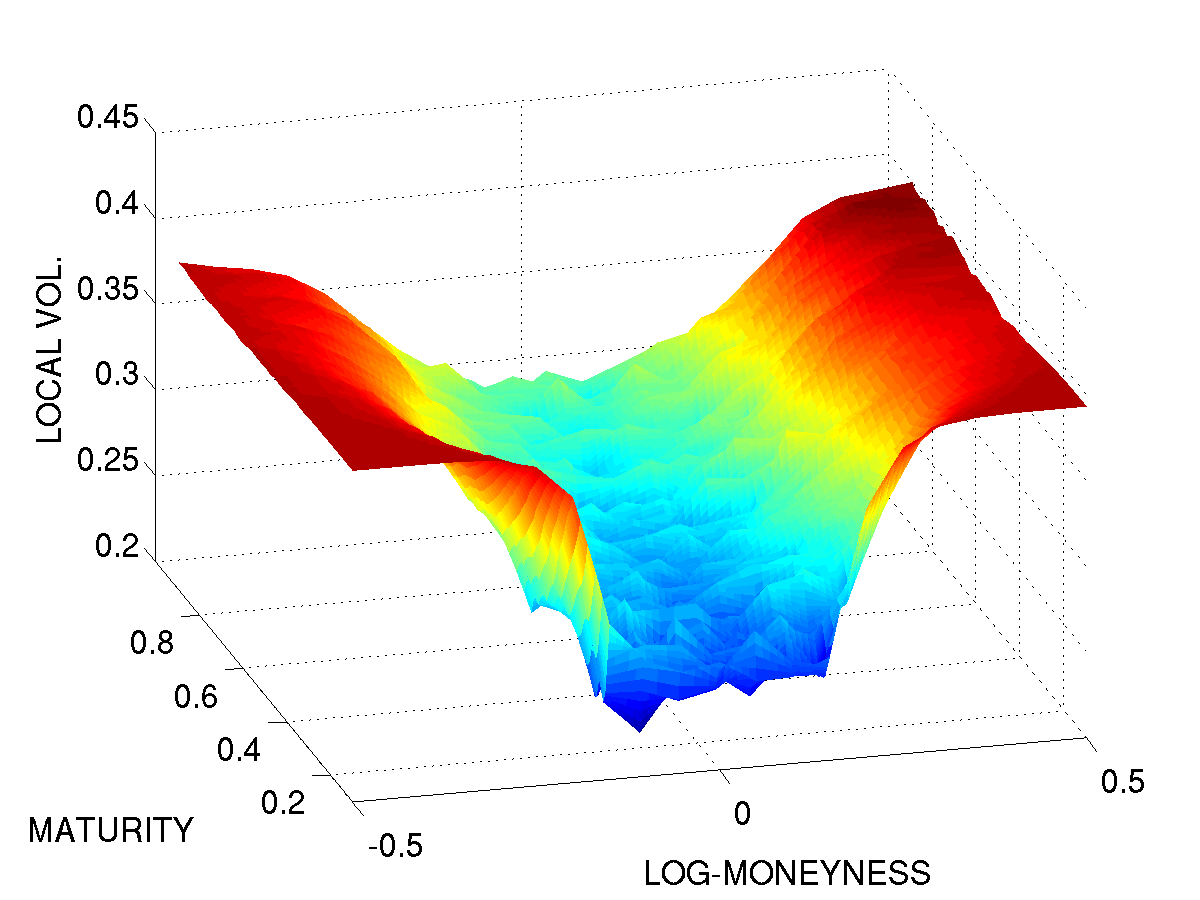}\hfill
\includegraphics[width=0.33\textwidth]{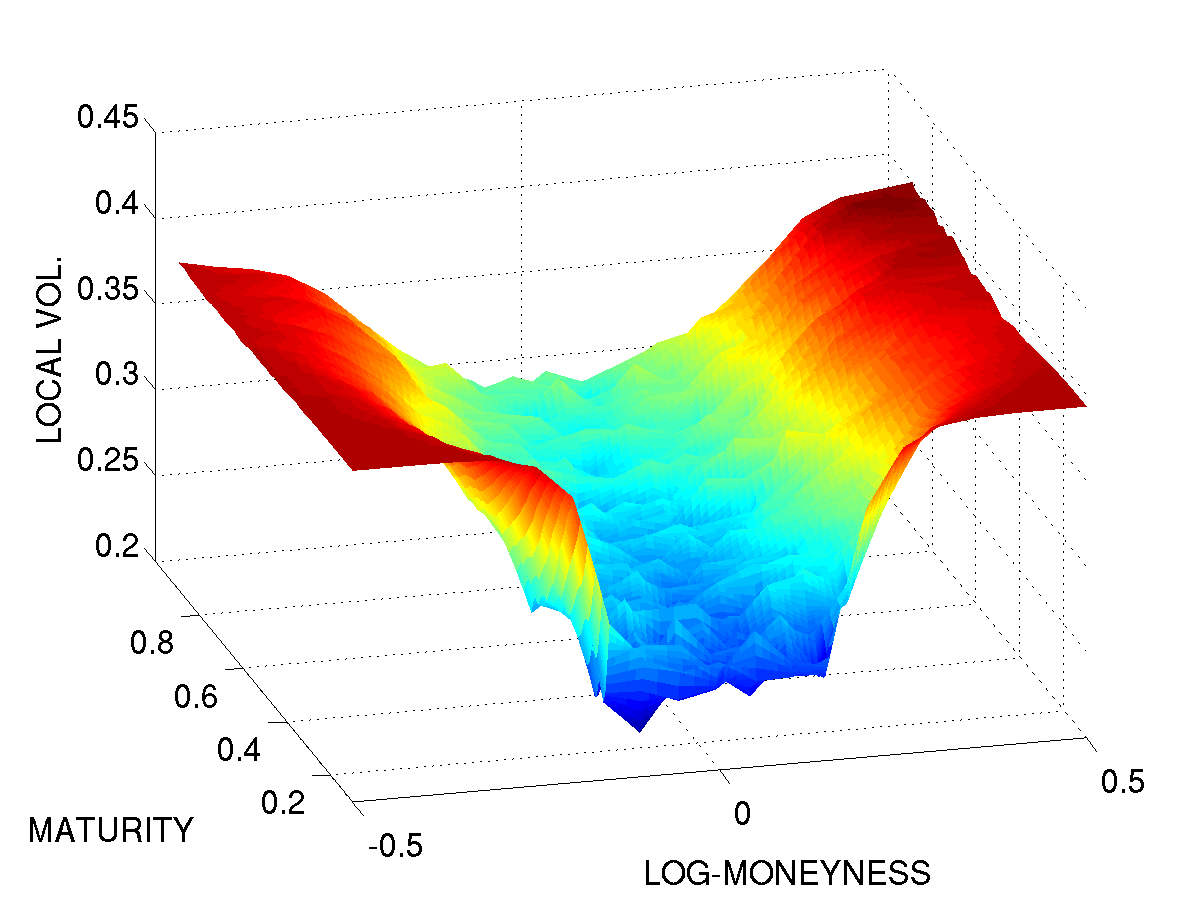}
\caption{Left: original surface. Center and right: reconstructions satisfying the discrepancy principle of Figure~\ref{test1}.}
\label{test4}
\noindent %\hrulefill
\end{figure}

%Figures~\ref{test1} and \ref{test2} presents the . 
We also calculate the $L^2$-error, i.e., the $L^2(D)$ distance between the regularized solution and the original local volatility surface. The resulting $L^2$-error for the regularized solutions used in Figure~\ref{test1} can be found in Figure~\ref{test2}. Note that, as expected, the  first two reconstructions satisfying the discrepancy principle above minimize the $L^2$-error, as we can see in Figures~\ref{test1} and \ref{test2}. It is an empirical evidence that the discrepancy principle in Equation~(\ref{discrepancy}) is a reliable way of finding the appropriate level of discretization of $\domain{F}$.

We used the standard regularizing functional
$$
f_{a_0}(a) = \|a - a_0\|^2_{H^1(D)},
$$
since it imposes smoothness in both time and space variables. As mentioned above, smoothness is an expected feature of local volatility. We also performed some tests using Kullback-Leibler regularization.

Figures~\ref{test3} and \ref{test4} present reconstructions of local volatility satisfying the discrepancy principle. Note that the reconstructions that first satisfy the discrepancy principle of Figure~\ref{test1} present better $L^2$-errors (as expected). Moreover, we can see that the surfaces of Figure~\ref{test3}, as it has a smaller number of unknowns, is smoother than the surfaces of Figure~\ref{test4}.%, whereby the larger time to maturity volatilities are flatter. 

\begin{figure}[ht]
\centering
\includegraphics[width=0.5\textwidth]{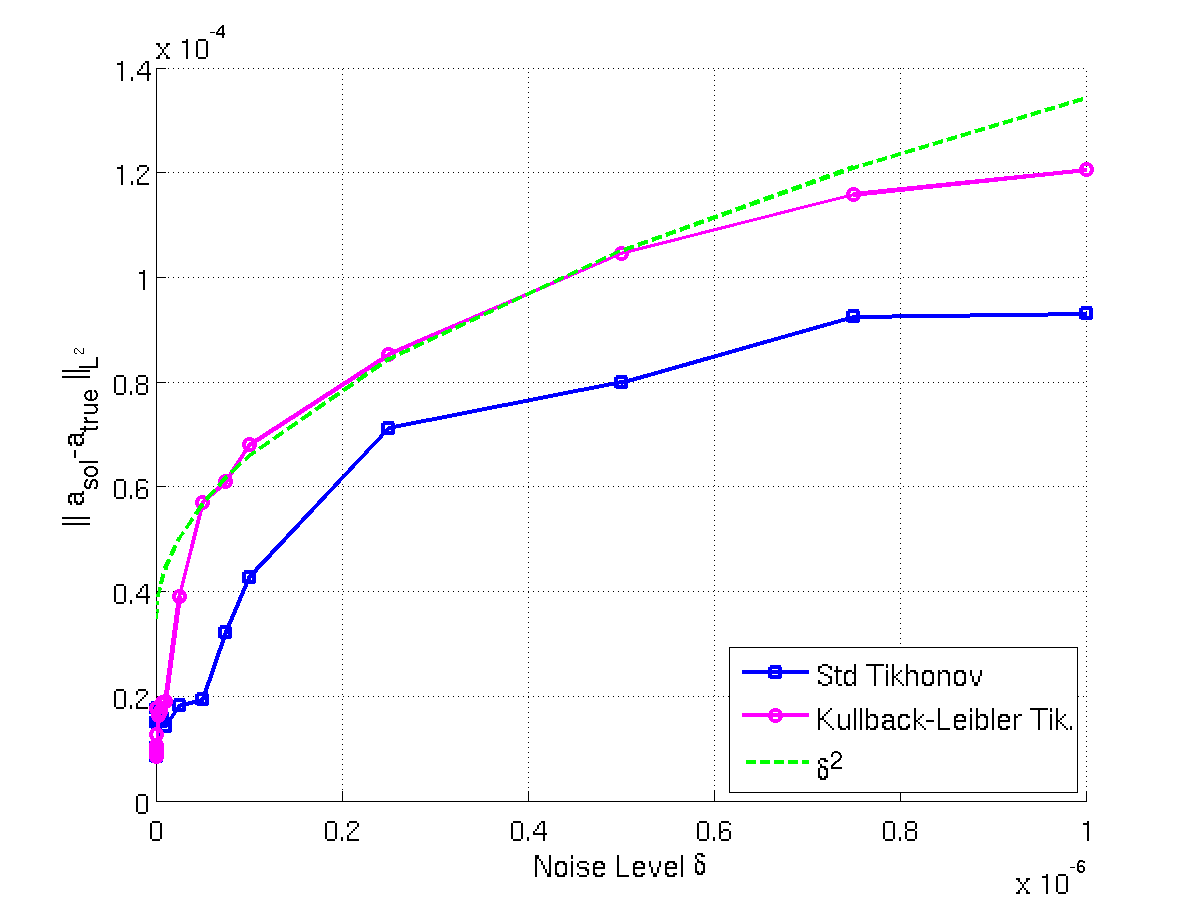}
\caption{The $L^2$-error $r(\delta) = \|a_{\beta}(\delta) - a^\dagger\|_{L^2(D)}$ as a function of the noise level $\delta$ satisfies the convergence rates of Theorem~\ref{th:convergence-rates1}. The dashed line is the function $g(\delta) = \sqrt{\delta}$.}
\label{figd}
\noindent %\hrulefill
\end{figure}

One illustration of Theorem~\ref{th:convergence-rates1} is given in Figure~\ref{figd}, whereby, for different convex regularization approaches, we calculated the $L^2$-error as a function of the noise level $\delta$, as it goes to zero. Note that the estimates of Theorem~\ref{th:convergence-rates1} are satisfied. In this specific example, for simplicity, we kept $\rho_m$ and $\gamma_n$ constant.%, since it was numerically solved in a fine mesh.

%------------------------------------------------------------------------------------------------
\subsection{Market Data: Henry Hub}\label{sec:hh}

We present below some reconstructions of the local volatility from call option prices on futures of the Henry Hub natural gas price traded in the CME stock exchange. In the present examples we use a uniform mesh with step sizes given by $\Delta t = 0.005$ and $\Delta y = 0.05$. We interpolate the data with a two-dimensional cubic spline. We also used a two-dimensional cubic spline basis for $\mathcal{D}(F)$ in the present set of experiments. We chose the appropriate discretization level for $\domain{F}$ by the discrepancy principle of Equation~(\ref{discrepancy}).

\begin{figure}[ht]
\centering
\includegraphics[width=0.45\textwidth]{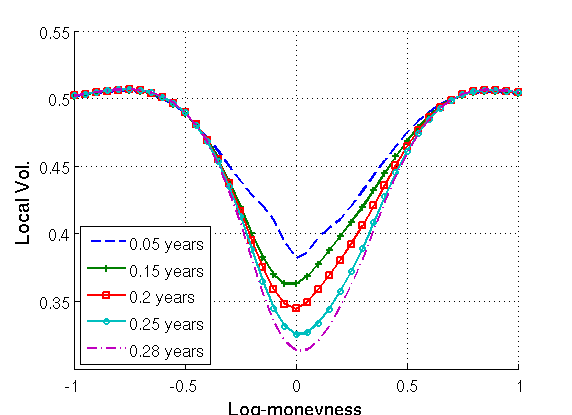}\hfill 
\includegraphics[width=0.45\textwidth]{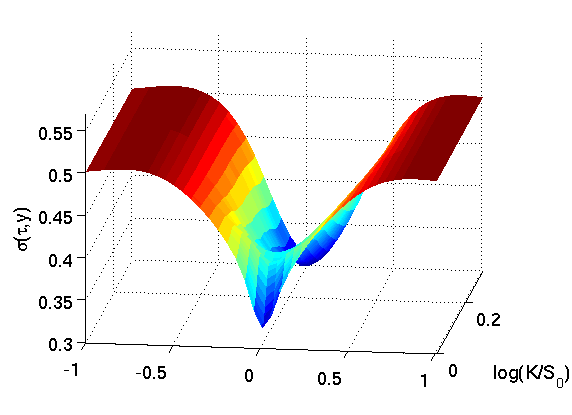}
\caption{Left: Reconstructed local volatility for some fixed maturities. Right: Reconstructed local volatility surface.}
\label{fig3}
\noindent %\hrulefill
\end{figure}

\begin{figure}[ht]
\centering
\includegraphics[width=0.45\textwidth]{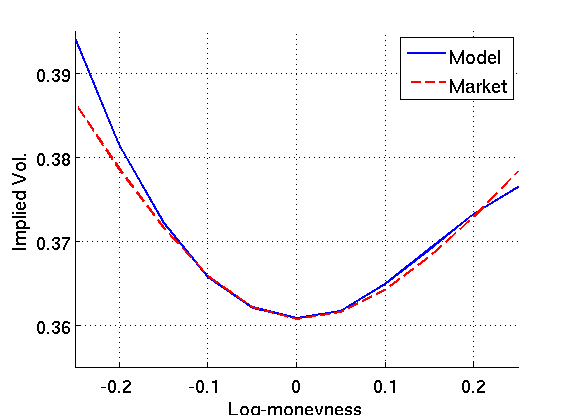}\hfill 
\includegraphics[width=0.45\textwidth]{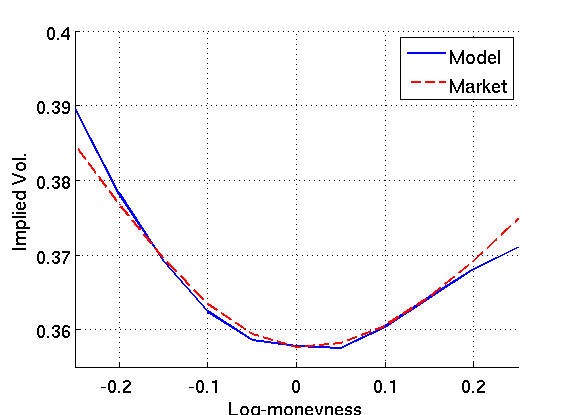}\hfill
\caption{Implied (Black-Scholes) Volatility for market prices (dashed line) and prices generated by the reconstructed volatility of Figure~\ref{fig3} (continuous line)for two different maturities.}
\label{fig3b}
\noindent %\hrulefill
\end{figure}

We used the standard $H^1(D)$ functional in these experiments with Tikhonov regularization, as in the above examples. It is important to note that, we used the traded price data as in \cite[Chapter 4]{vvlathesis}. In the latter, the authors have assumed that option prices were given as a function of the unknown commodity spot price, instead of the underlying future prices. 

Figure~\ref{fig3} presents the reconstructed local volatility for call option prices on Henry Hub future prices. Observe that the required smoothness is satisfied by the reconstructions. Figure~\ref{fig3b} presents the implied volatility for the traded prices and the prices generated by the reconstructed local volatility surface of Figure~\ref{fig3} for two different maturities. Note that they become more similar as we get closer to the so-called ``at-the-money'' values, i.e., $y = 0$. This turns out to be the region of the most liquid prices and thus less subjected to noise.

%------------------------------------------------------------------------------------------------
\subsection{Market Data: S\&P 500}\label{sec:spx}

The reconstructions of local volatility presented in Figures~\ref{fig4}, \ref{fig5}, \ref{fig6}, \ref{fig7} and \ref{fig8} were obtained from call option prices on the S\&P500 index traded at the NYSE on 05/09/2013. We used a uniform mesh with step sizes given by $\Delta t = 0.005$ and $\Delta y = 0.05$. We interpolate the data with the method proposed by N.~Kahale in \cite{kahale}. The main feature of this method is to avoid arbitrage opportunities by keeping the option prices convex functions of the strike. We also used a two-dimensional linear basis for $\mathcal{D}(F)$ in the present set of experiments. We chose the appropriate discretization level for $\domain{F}$ by the discrepancy principle of Equation~(\ref{discrepancy}). Again, we used the standard $H^1(D)$ functional in the Tikhonov regularization.

\begin{figure}[ht]
\centering
\includegraphics[width=0.33\textwidth]{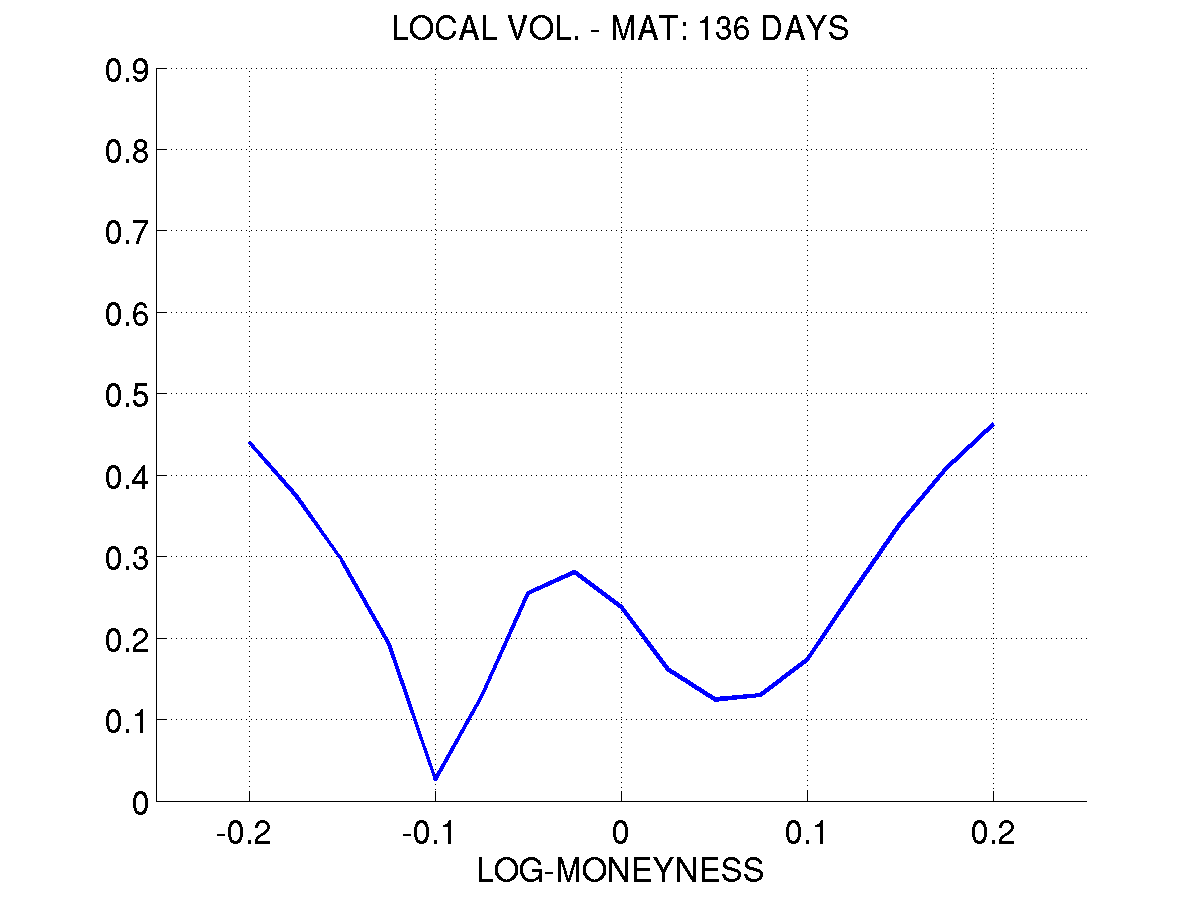}\hfill 
\includegraphics[width=0.33\textwidth]{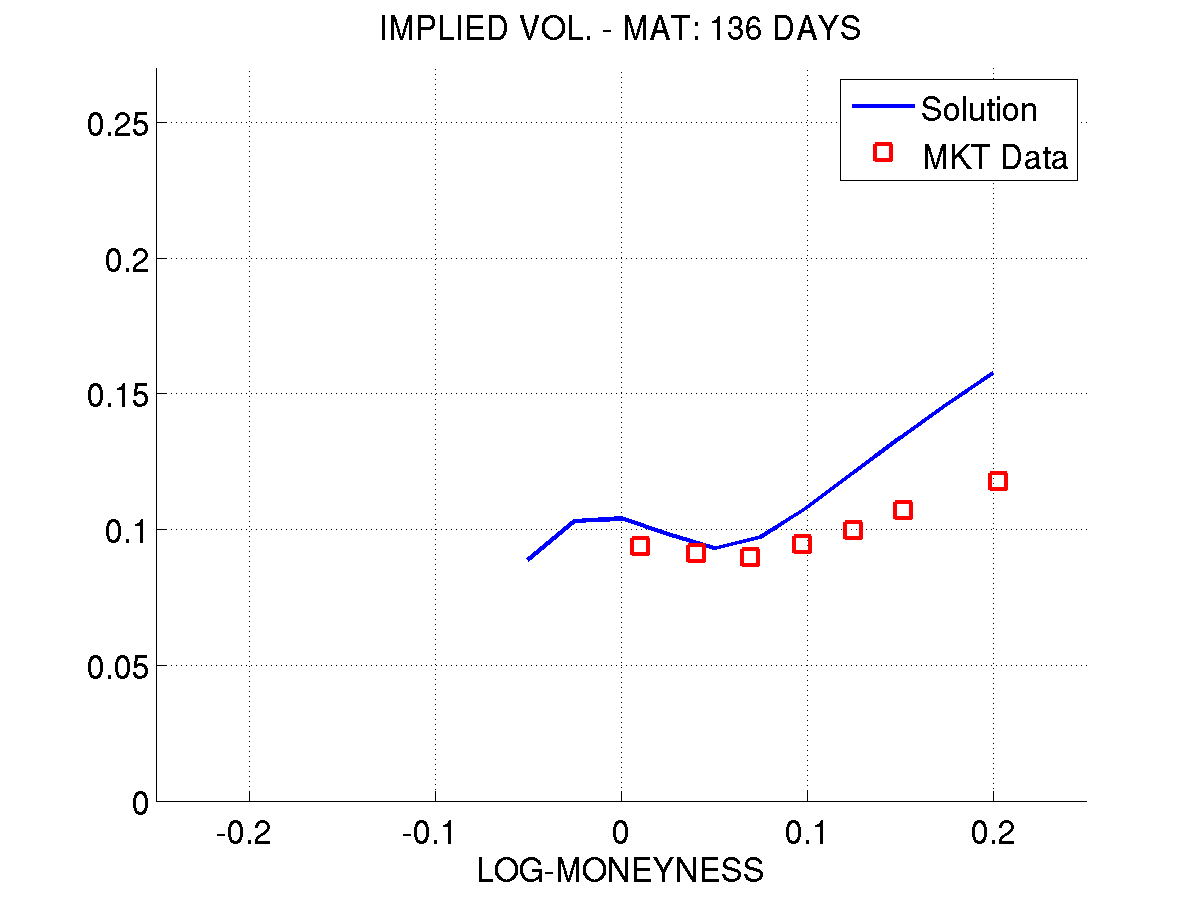}\hfill
\includegraphics[width=0.33\textwidth]{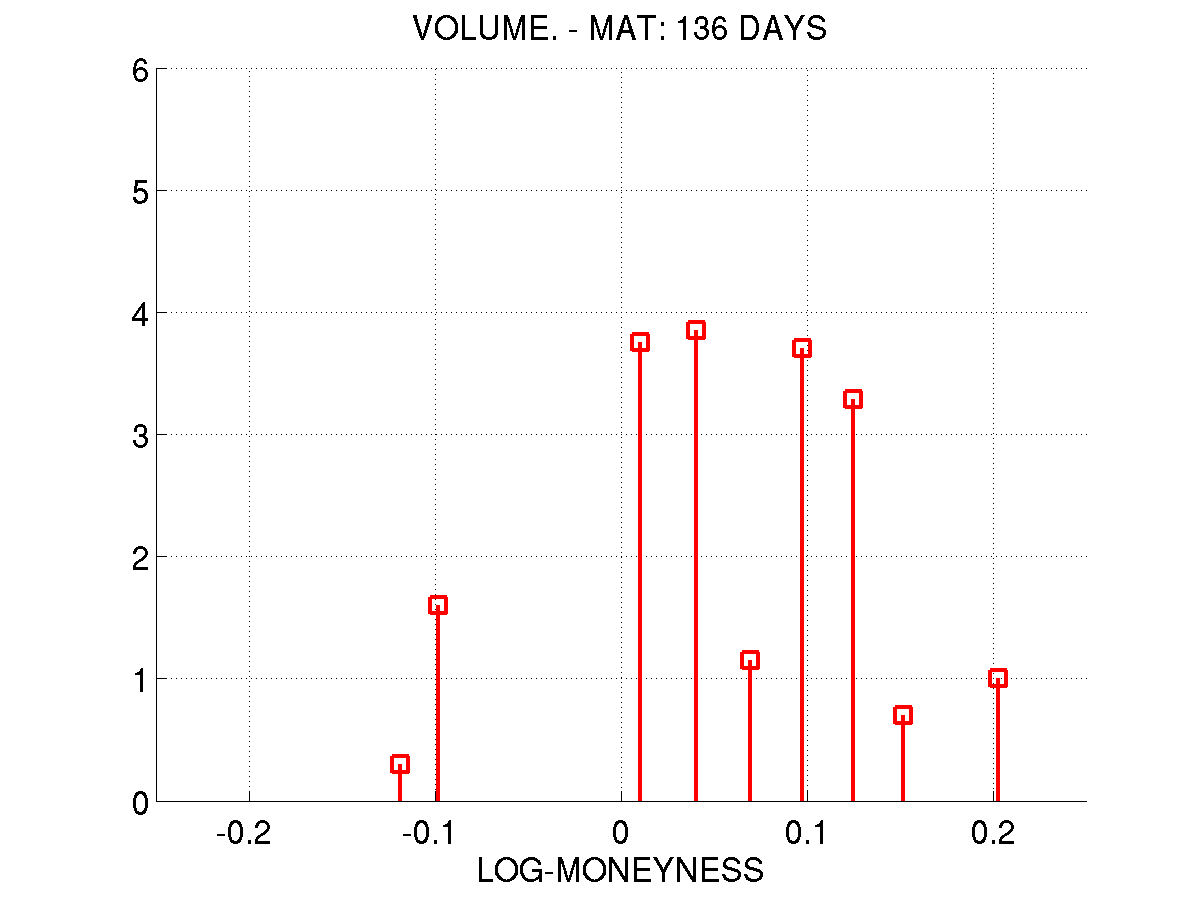}
\caption{Left: Local volatility reconstructed from S\&P500 call option prices traded on 05/09/2013 maturing on 09/21/2013. Center: Implied volatility of market data (squares) and implied volatility of the local volatility of the first figure (continuous line). Right: Volume of the prices used in this reconstruction in a $\log_{10}$ scale.}
\label{fig4}
\noindent %\hrulefill
\end{figure}

\begin{figure}[ht]
\centering
\includegraphics[width=0.33\textwidth]{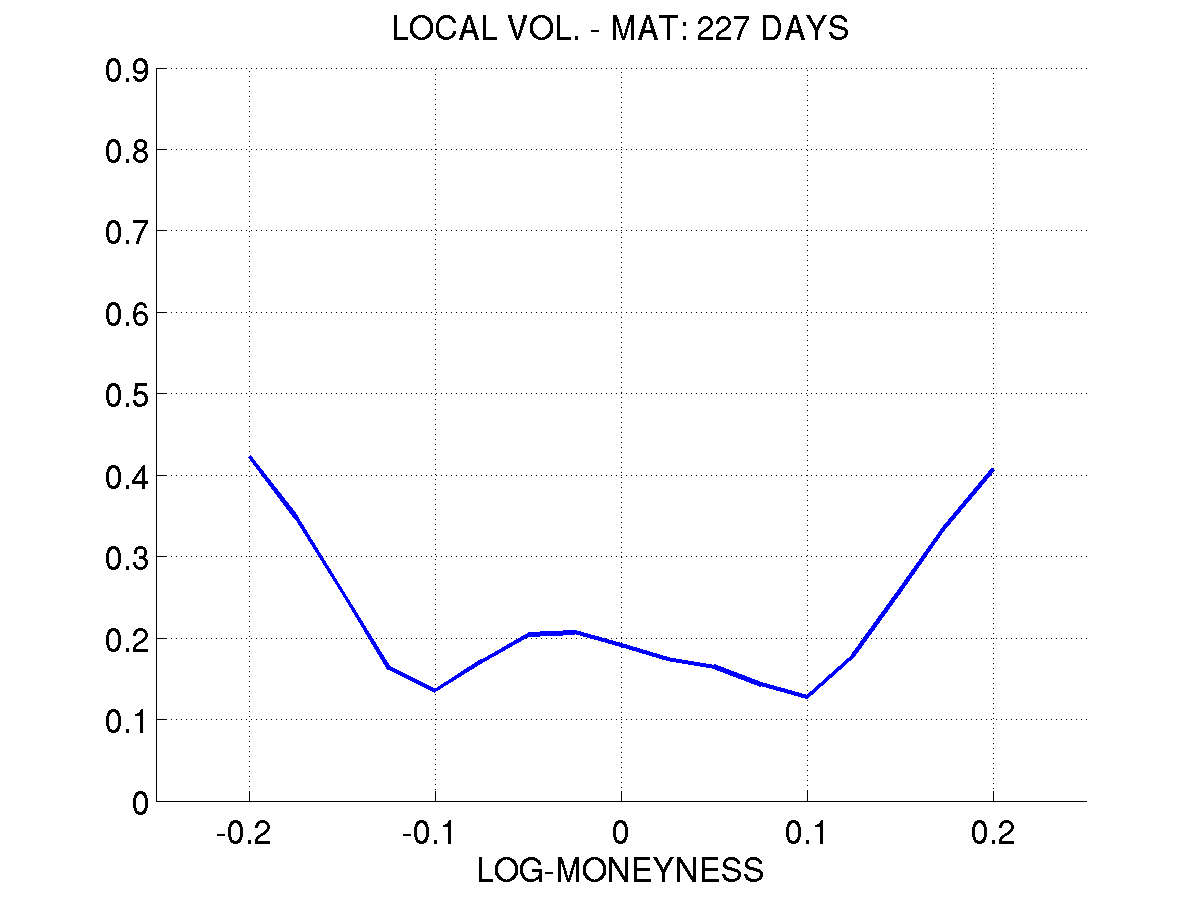}\hfill 
\includegraphics[width=0.33\textwidth]{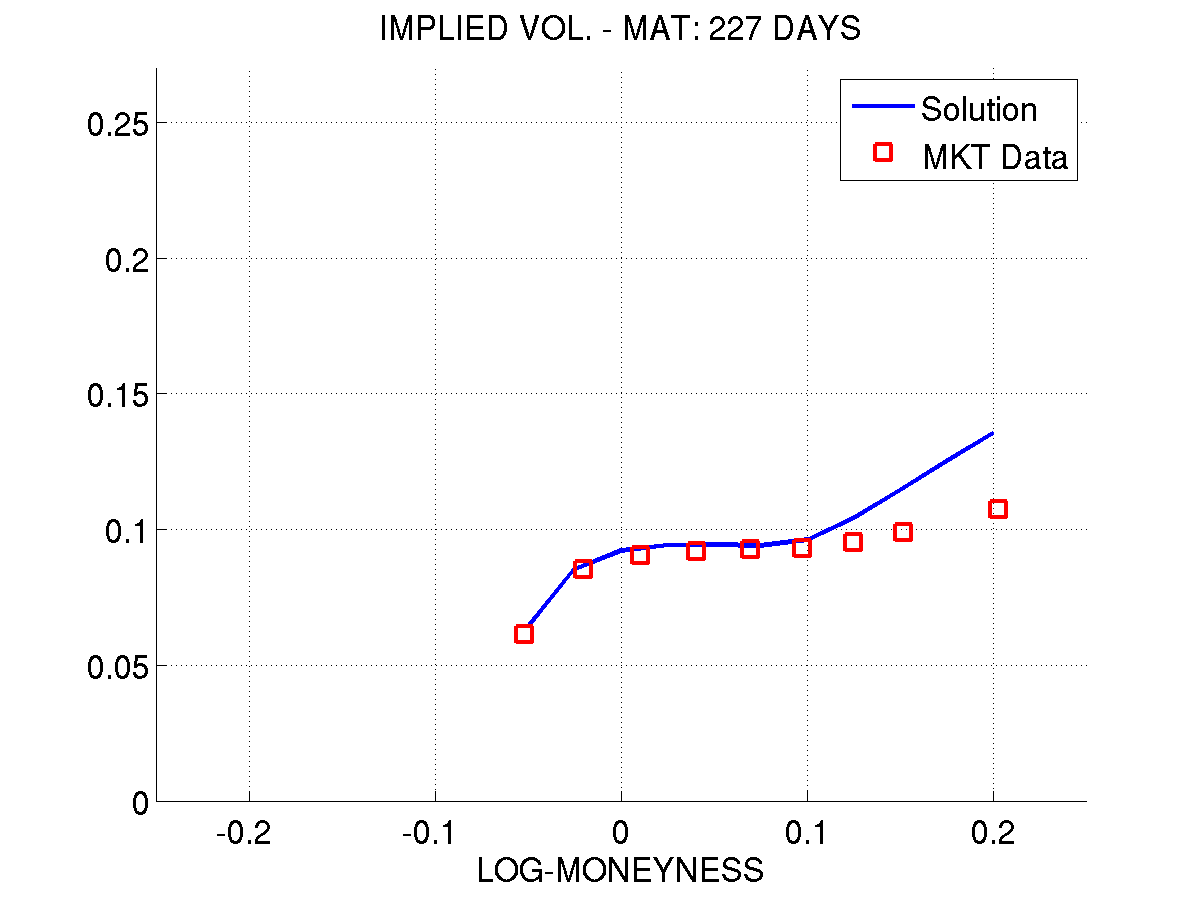}\hfill
\includegraphics[width=0.33\textwidth]{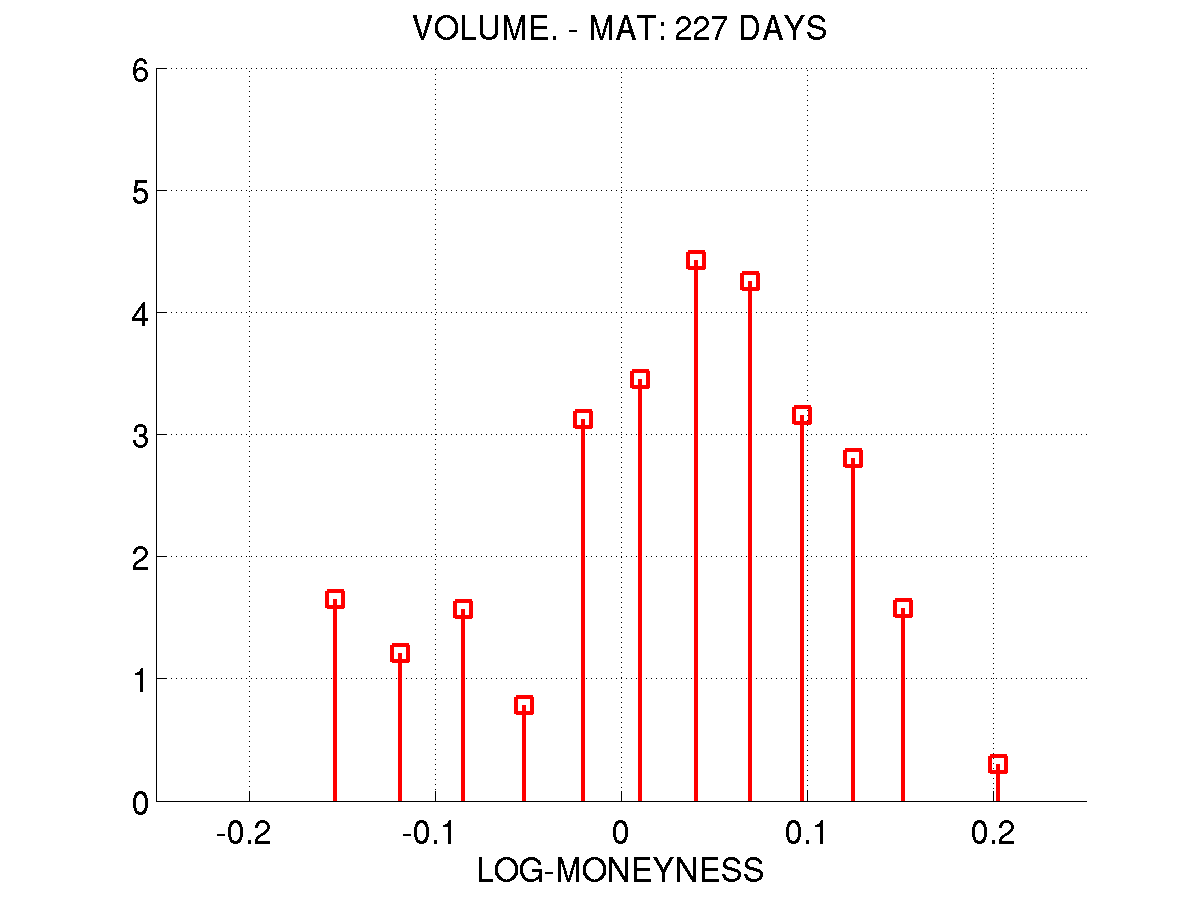}
\caption{Left: Local volatility reconstructed from S\&P500 call option prices traded on 05/09/2013 maturing on 12/21/2013. Center: Implied volatility of market data (squares) and implied volatility of the local volatility of the first figure (continuous line). Right: Volume of the prices used in this reconstruction in a $\log_{10}$ scale.}
\label{fig5}
\noindent %\hrulefill
\end{figure}

\begin{figure}[ht]
\centering
\includegraphics[width=0.33\textwidth]{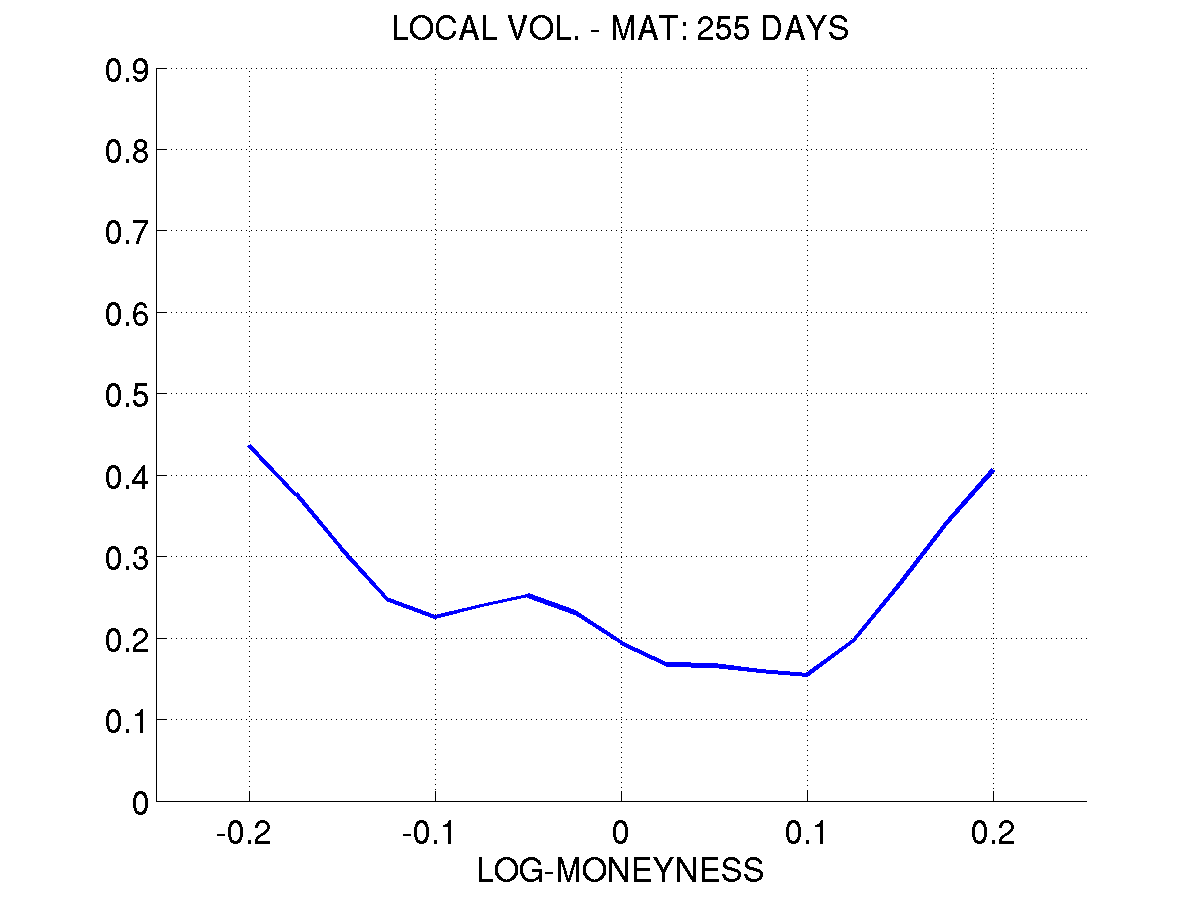}\hfill 
\includegraphics[width=0.33\textwidth]{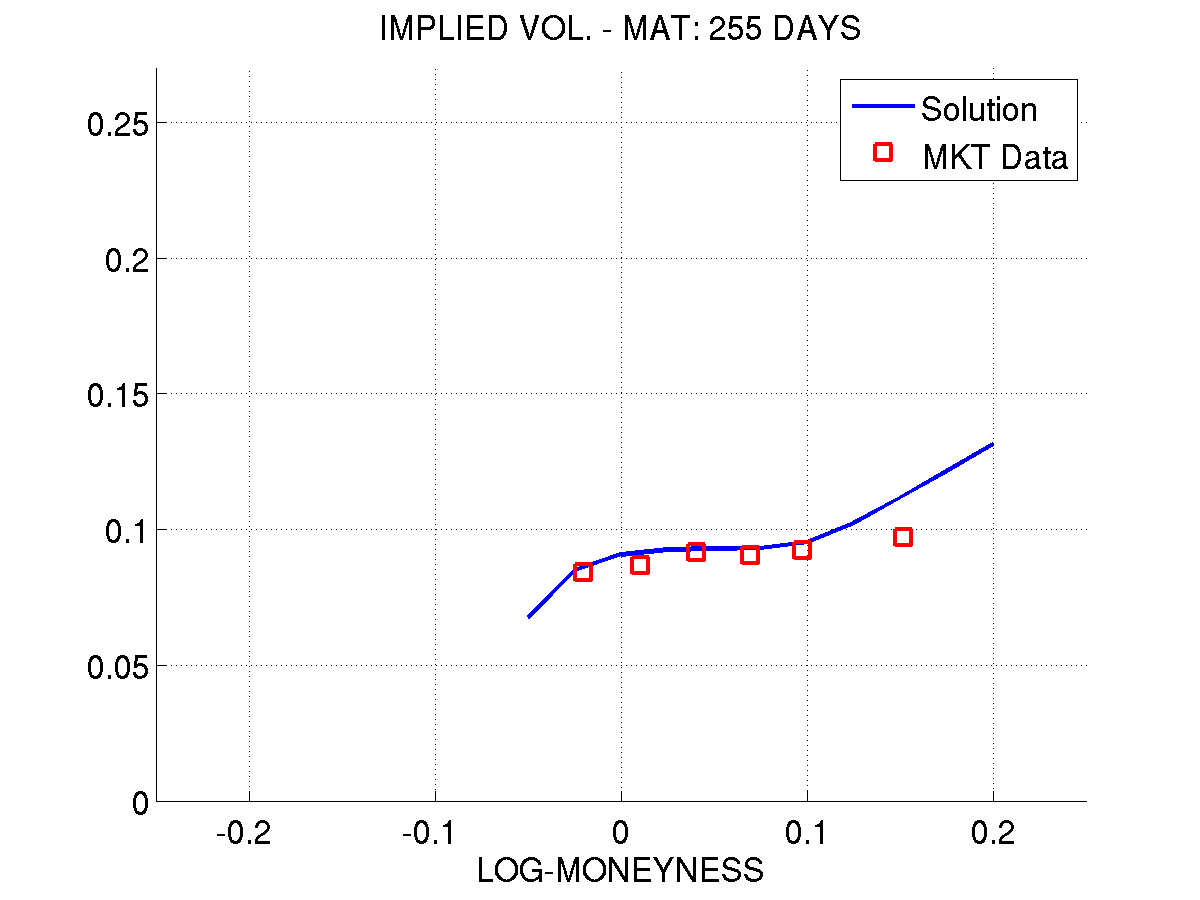}\hfill
\includegraphics[width=0.33\textwidth]{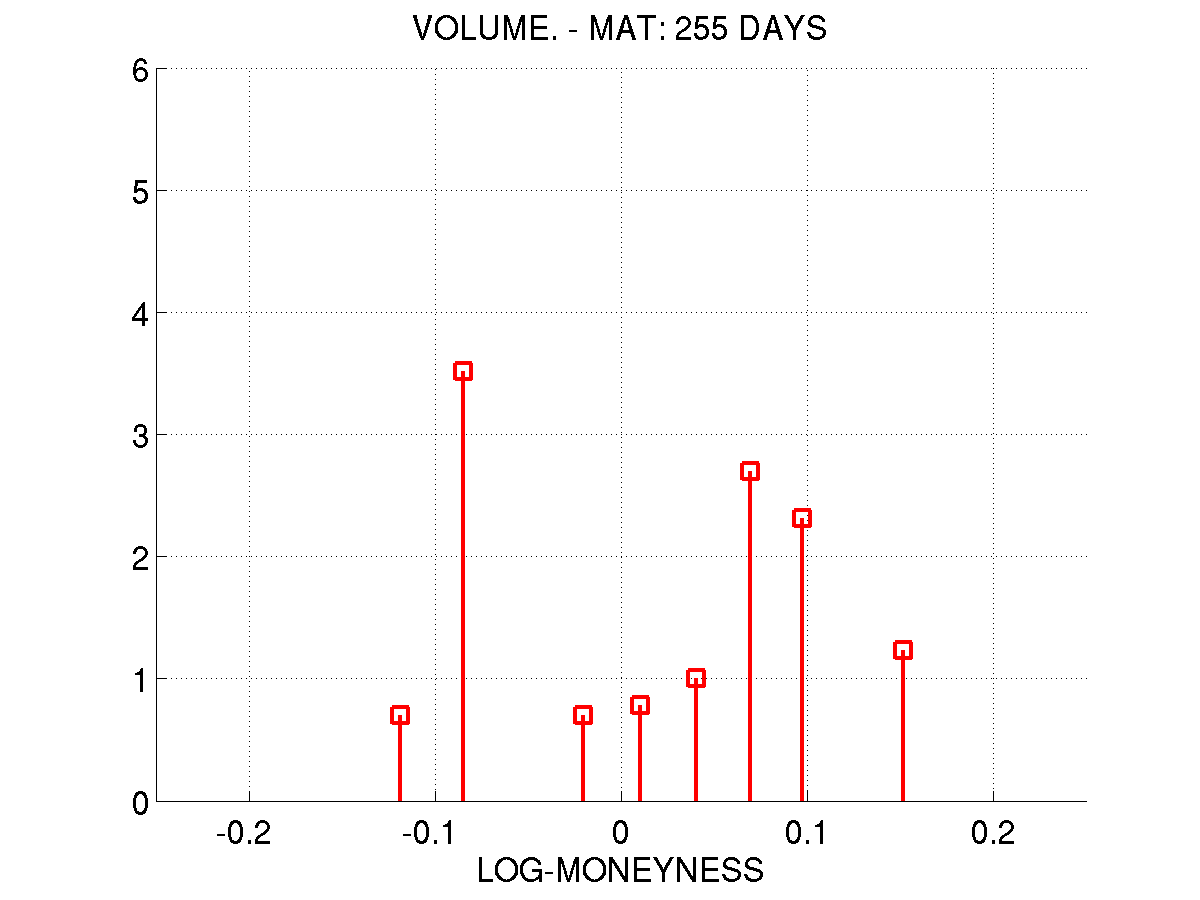}
\caption{Left: Local volatility reconstructed from S\&P500 call option prices traded on 05/09/2013 maturing on 01/18/2014. Center: Implied volatility of market data (squares) and implied volatility of the local volatility of the first figure (continuous line). Right: Volume of the prices used in this reconstruction in a $\log_{10}$ scale.}
\label{fig6}
\noindent %\hrulefill
\end{figure}

\begin{figure}[ht]
\centering
\includegraphics[width=0.33\textwidth]{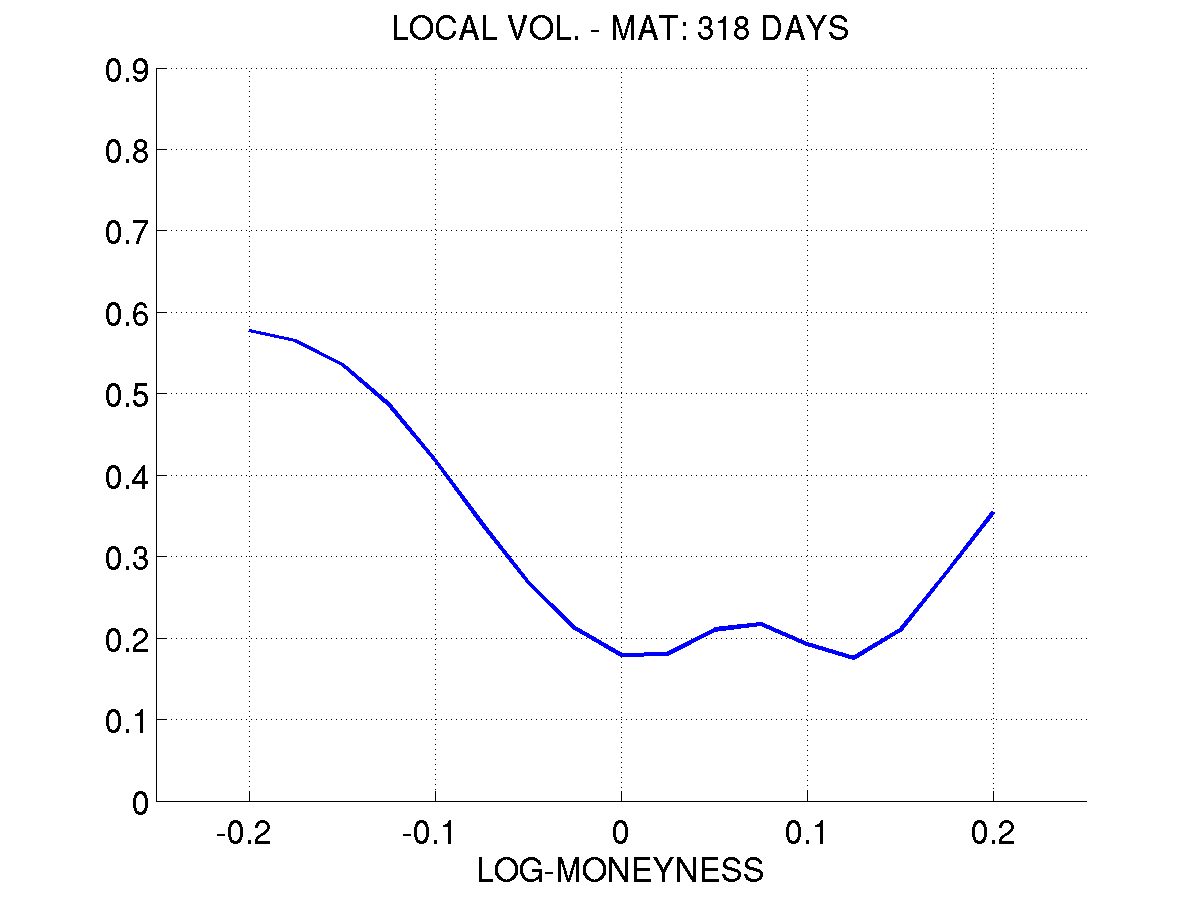}\hfill 
\includegraphics[width=0.33\textwidth]{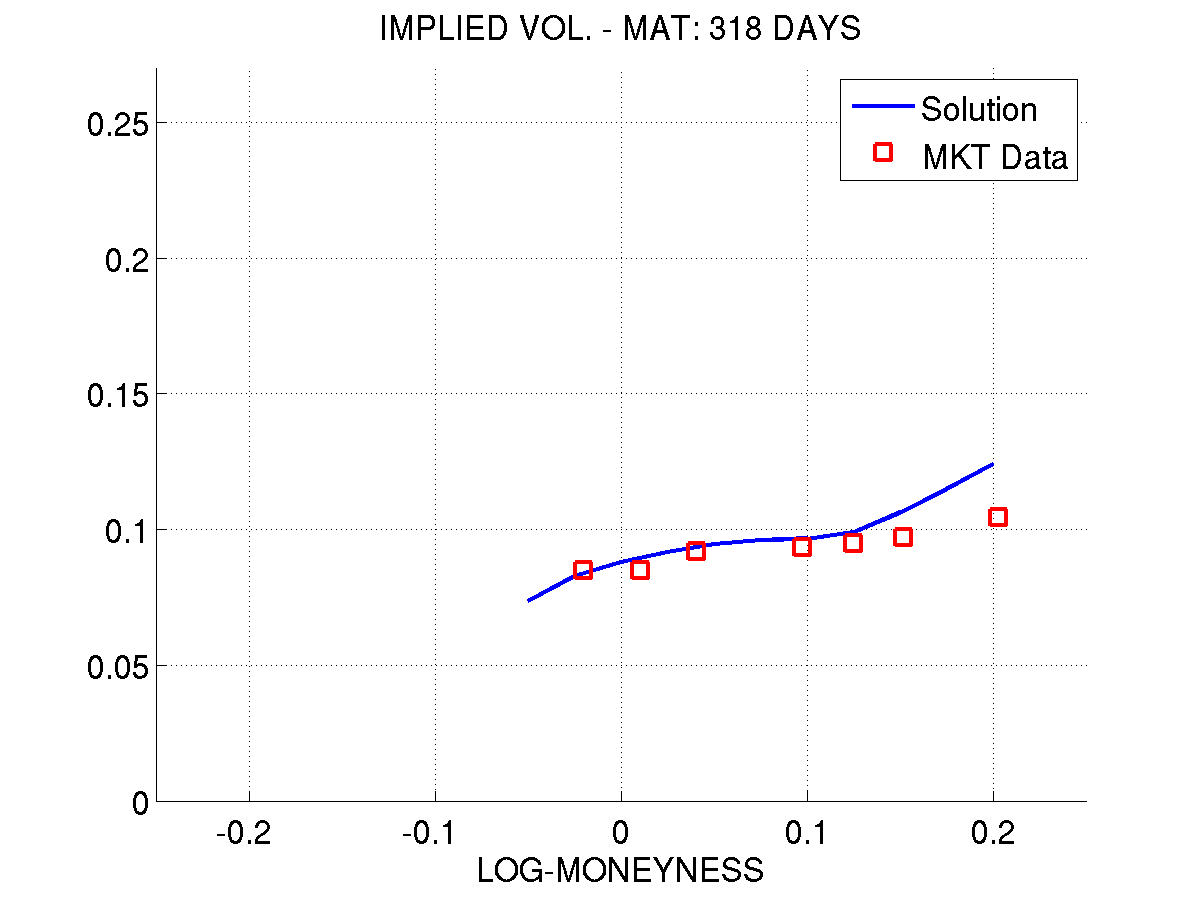}\hfill
\includegraphics[width=0.33\textwidth]{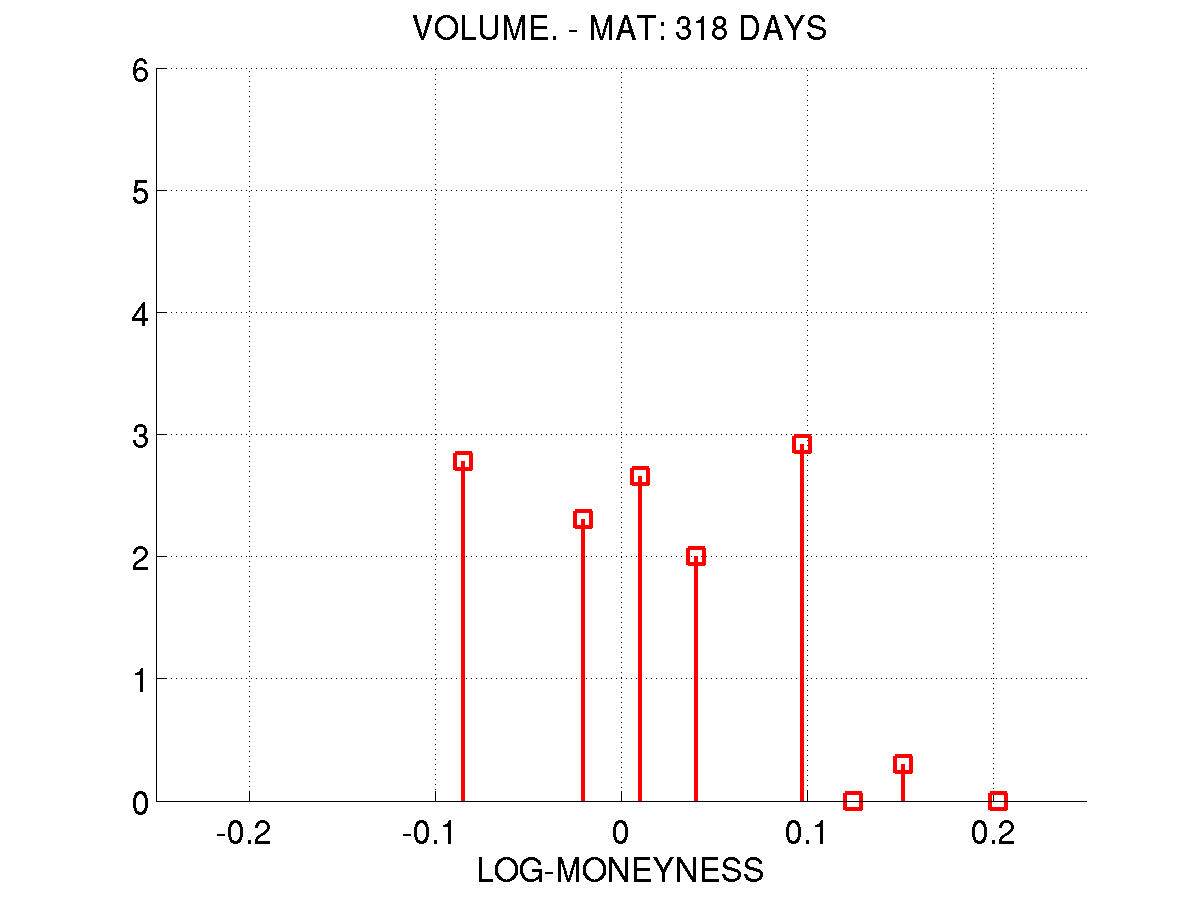}
\caption{Left: Local volatility reconstructed from S\&P500 call option prices traded on 05/09/2013 maturing on 03/22/2014. Center: Implied volatility of market data (squares) and implied volatility of the local volatility of the first figure (continuous line). Right: Volume of the prices used in this reconstruction in a $\log_{10}$ scale.}
\label{fig7}
\noindent %\hrulefill
\end{figure}

\begin{figure}[ht]
\centering
\includegraphics[width=0.33\textwidth]{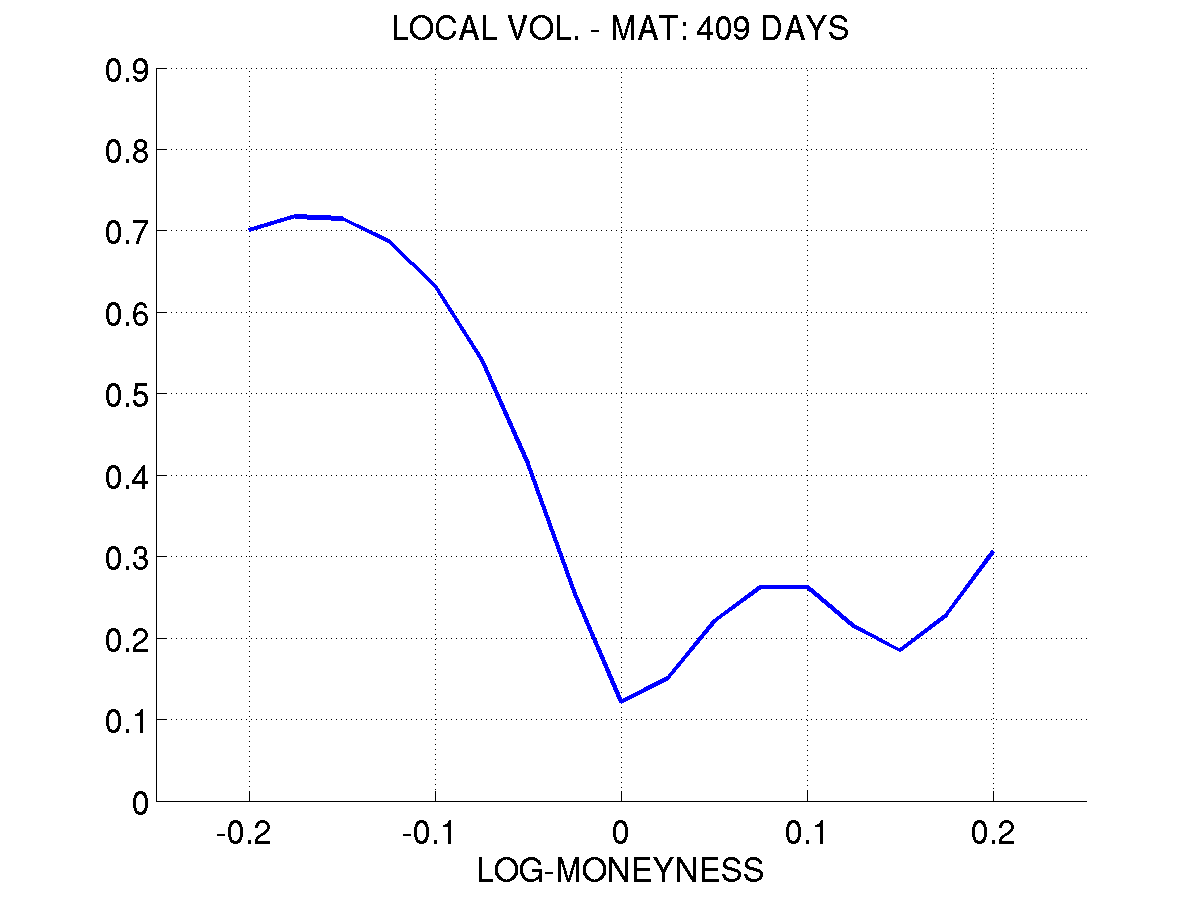}\hfill 
\includegraphics[width=0.33\textwidth]{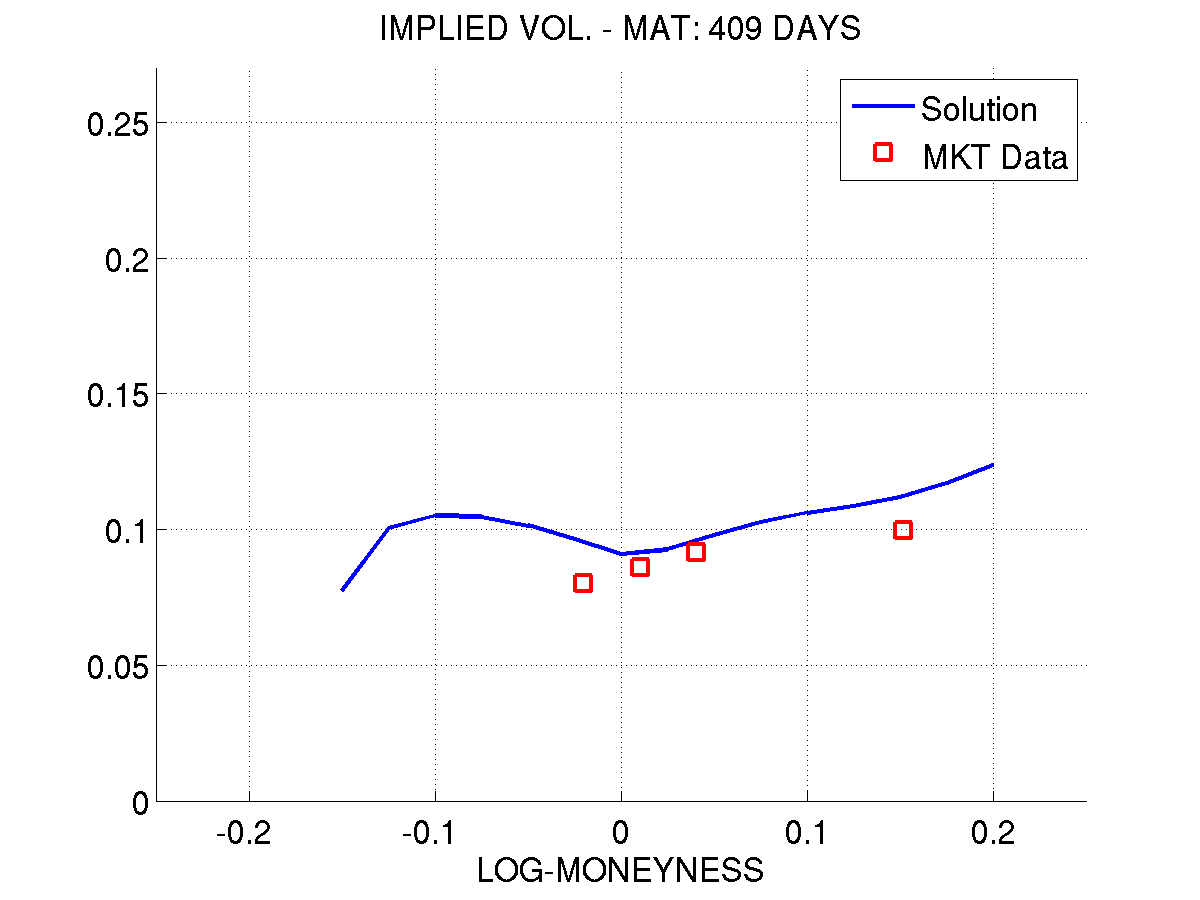}\hfill
\includegraphics[width=0.33\textwidth]{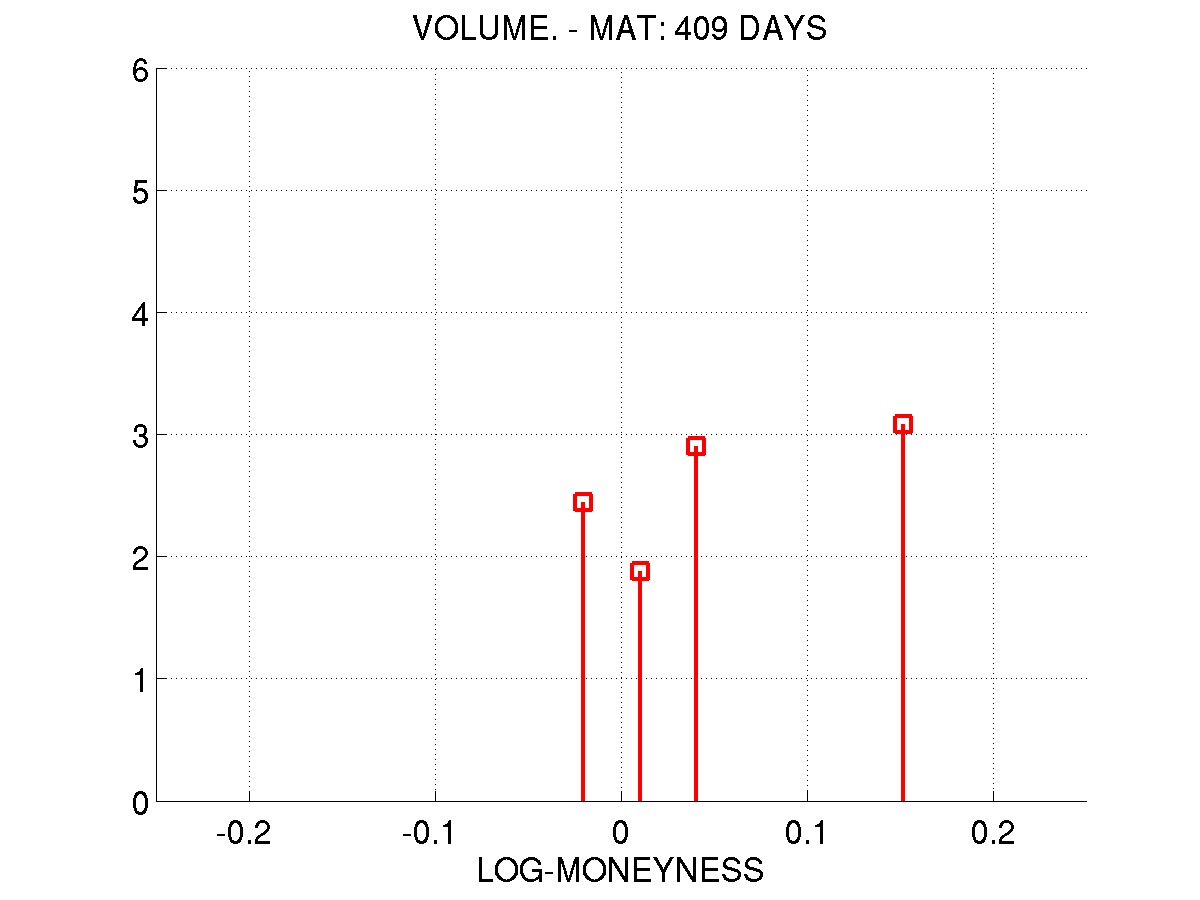}
\caption{Left: Local volatility reconstructed from S\&P500 call option prices traded on 05/09/2013 maturing on 06/21/2014. Center: Implied volatility of market data (squares) and implied volatility of the local volatility of the first figure (continuous line). Right: Volume of the prices used in this reconstruction in a $\log_{10}$ scale.}
\label{fig8}
\noindent %\hrulefill
\end{figure}

The left picture in Figures~\ref{fig4}, \ref{fig5}, \ref{fig6}, \ref{fig7} and \ref{fig8}, presents the calibrated local volatility curves from S\&P500 call option data for different maturities. Although we have used a piecewise linear basis, the required smoothness was satisfied by the reconstructions. The center picture presents the implied volatility of market data (squares) and implied volatility of the prices related to the local volatility of the left picture (continuous line). Note that these two implied volatility are very close, specially when the trading volume of prices, presented in the right picture, is high, i.e., the prices are more liquid. 

Note that, intuitively, more liquid prices are less subject to noise introduced by trading. Then, such numerical results show a strong agreement with our theoretical results of previous sections.

%-----------------------------------------------------------------------------------------------
\section{Conclusions and Final Remarks}
Although our approach emphasizes the discreteness of the data and of
the domain where the reconstruction is taking place, it does not rely
on a binomial tree approach or a discrete stochastic process. Thus the
approach differs substantially from attempts to calibrate the
transition probability of the binomial tree of previous works such as
\cite{jackwerth}

As mentioned in the introduction, one of the novelties of our approach to this subject is 
that of 
establishing convergence and convergence rate results of the discrete approximation. Such a result allows us to take into account different levels of uncertainties in the traded prices and convert it into estimates for the choice of the discretization level of the volatility calibration. 

In particular, our theoretical and numerical results show that there exists an intrinsic connection between different levels of uncertainty in the data and the corresponding calibrated volatility surface. Furthermore, if the uncertainty level goes to zero,  then the calibrated volatility converges to the true volatility. Finally, if we have an upper bound on the financial data uncertainty, then there exists an optimal choice for the discretization level of the volatility surface. Although we do not present a theoretical proof, we proposed a discrepancy principle to determine the optimal discretization level. 
This claim is illustrated by the results of Section~\ref{sec:synthetic}.
The theoretical aspects concerning the implementation of this discrepancy principle are beyond the scope of the present work and shall be discussed in \cite{ADZ2013}.

We also presented reconstructions of the local volatility surface with market data. Using  S\&P500 data as a validation criteria, we compared the implied volatility obtained from traded data and the prices given by the model. 
This was  compared to the volume information in the validation of results. As expected, 
the more liquid the traded contracts, the better the results. In other words, the implied volatilities almost coincided for highly traded option strikes and maturities. 
These results can be found in Section~\ref{sec:spx}.

We also applied our technique to Henry Hub natural gas data. In this case,  we used the same criteria for validating the results, but we did not have volume information. However, for prices near or at the money, we also had a good match of the implied volatilities of market and model prices. This is justified by the natural higher liquidity of the prices around the at-the-money level. These results can be found in Section~\ref{sec:hh}.

\section*{Acknowledgements}
V.A. acknowledges and thanks CNPq, Petroleo Brasileiro S.A. and Ag\^encia Nacional do Petr\'oleo for the financial support during the time when this work was developed.
ADC. acknowledges and thanks the financial support from CNPq SwB grant 200815/2012-1, and from ARD FAPERGS grant 0839 12-3.
J.P.Z. acknowledges and thanks the financial support from CNPq through grants 302161/2003-1 and
474085/2003-1, and from FAPERJ through the programs {\em Cientistas do Nosso Estado} and {\em Pensa Rio}.

%-----------------------------------------------------------------------------------------------
\bibliographystyle{siam}%{amsplain}%{abbrv}
%\bibliography{lastedVersion}

%-----------------------------------------------------------------------------------------------
%%%%%%%%%%%%%%%%%%%%%%%%%%%%%%%%%%%%%%%%%%%%%%%%%%%%%%%%%%%%%%%%%%%%%%
\appendix
\section{Properties of the forward operator and ill-posedness of the inverse problem}\label{sec:2}
We now summarize some properties of the operator $F$ that have appeared in the literature \cite{DeCezaroScherzerZubelli09, Crepey-2003,
Egger-Engl2005,Hof-Kra-2005, HKPS-2007}. Such properties were used to prove some regularizing
aspects of the approximate solutions of the inverse problem under consideration.
%%%%%%%%%%%%%%%%%%%%%%%%%%%%% THEOREM
\begin{theo}\label{th:prop-F}
\begin{enumerate}
\item\label{item1} The operator $F\,:\,\domain{F} \subset H^{1+\varepsilon}(D)
\longrightarrow W^{1,2}_2(D)$ is compact. Moreover, $F$ is
weakly (sequentially) continuous and thus weakly closed.

\item\label{item2} Let $D$ be a bounded subset of $\mathbb{R}^2$ with
Lipschitz boundary. Moreover, let $a_n \in \mathcal{D}(F)$ with $a_n
\rightharpoonup a$ in $\Y$. Then $F(a_n) \rightharpoonup F(a)$ in
$\Y$.

\item\label{item3} $F$  is differentiable at $a \in \domain{F}$ in the direction $h$
such that $a+h \in  \domain{F}$. The derivative $F'(a)$ satisfies
\begin{align}\label{eq:directional_deriv}
-(u'\cdot h)_{\tau}+a((u'\cdot h)_{yy}-(u'\cdot h)_y)=-h(u_{yy}-u_y)
\end{align}
with homogeneous boundary and initial conditions. Also, $F'(a)$ is
extendable to a bounded linear operator on
$H^{1+\varepsilon}(D)$, i.e.,
\begin{align}\label{eq:C-limitation-der-F}
||F'(a)h||_{W^{1,2}_2(D)}\leq
C||h||_{H^{1+\varepsilon}(D)}.
\end{align}
 Moreover, $F'(a)$ satisfies the Lipschitz condition
\begin{align}\label{eq:Lipschitz-der-F}
\norm{F'(a) -
F'(a+h)}_{\mathcal{L}(H^{1+\varepsilon}(D);W^{1,2}_2(D))}
\leq c\norm{h}_{H^{1+\varepsilon}(D)}\,,
\end{align}
for $a+h \in \domain{F}$.
\end{enumerate}
\end{theo}
%%%%%%%%%%%%%%%%%%%%%%%%%% END THEOREM
%
%%%%%%%%%%%%%%%%%%%%%%%%%% PROOF
\begin{proof}
Items~\eqref{item1} and \eqref{item3} were proved in
\cite{DeCezaroScherzerZubelli09, Egger-Engl2005}. Item~\ref{item2}
was proved in \cite{DeCezaroScherzerZubelli09}.
\end{proof}
%%%%%%%%%%%%%%%%%%%%%%%%% END PROOF
%
%%%%%%%%%%%%%%%%%%%%%%%%% REMARK
\begin{remark} An immediate consequence of the compactness and weak closedness of the operator $F$
(see Theorem \ref{th:prop-F}) is the local ill-posedness
of the volatility calibration problem. See \cite{DeCezaroScherzerZubelli09, Egger-Engl2005}. 
\end{remark}
%%%%%%%%%%%%%%%%%%%%%%% END REMARK

The next lemma was already known in \cite{DeCezaroScherzerZubelli09},
however, we present it here for sake of completeness.
%
%%%%%%%%%%%%%%%%%%%%%%% LEMMA
\begin{lemma}\label{lemma:F-dif-inective}
The Fr\'echet derivative of the operator $F$,
$$
F'(a):H^{1+\varepsilon}(D) \longrightarrow W^{1,2}_2(D)\,,
$$
is injective and compact.
\end{lemma}
%%%%%%%%%%%%%%%%%%%%%%% END LEMMA

A consequence of Lemma \ref{lemma:F-dif-inective} is the fact that $F$ cannot 
be constant along any affine subspace through $a$ and
parallel to $\mathcal{N}(F'(a))$. Hence, the \textit{tangential cone
condition}, presented in the theorem below, does not represent a 
severe assumption in the present context. See the comments in \cite[Chapter
11]{EngHanNeu96}. The proof of the following theorem can be found in
\cite{DCZ2010, DCZ2011}.

\begin{theo}
The parameter-to-solution map $F\,:\,\domain{F}\subset
H^{1+\varepsilon}(D) \longrightarrow W^{1,2}_2(D)$
satisfies the local \textit{tangential cone condition}, there exists $0 < \eta <1/2$ s.t.
\begin{align}\label{eq:tg-cone-condiction}
\norm{F(a) - F(\tilde{a}) -
F'(\tilde{a})(a-\tilde{a})}_{W^{1,2}_2(D)}\leq \eta \norm{F(a)
- F(\tilde{a})}_{W^{1,2}_2(D)}\,,%\qquad \eta < \frac{1}{2}\,,
\end{align}
for all $a , \tilde{a}$ in a ball $B_\rho(a^*)\subset
\mathcal{D}(F)$ with some $\rho >0$.

In particular, there exists $0< \widetilde{\eta} < 1/2$ s.t.
\begin{align}\label{eq:tg-cone-condiction1}
\norm{F(a) - F(\tilde{a}) -
F'(\tilde{a})(a-\tilde{a})}_{W^{1,2}_2(D)}\leq \widetilde{\eta}
\norm{a - \tilde{a}}_{\X} \norm{F(a) -
F(\tilde{a})}_{W^{1,2}_2(D)}\,,%\qquad \eta < \frac{1}{2}\,,
\end{align}
for all $a , \tilde{a}$ in a ball $B_\rho(a^*)\subset
\mathcal{D}(F)$ with some $\rho >0$.
\label{pr:cone-cond}
\end{theo}
%-----------------------------------------------------------------------------------------------
\end{document}